\documentclass[12pt,a4paper,twoside]{amsart}
\usepackage{amscd,amssymb,amsmath,amsfonts,amsthm}
\usepackage{a4}
\usepackage{color}

\usepackage{enumitem}

\usepackage[utf8]{inputenc}
\usepackage[english]{babel}

\newtheorem{theorem}{Theorem}[section]

\newtheorem{lemma}[theorem]{Lemma}

\newtheorem{proposition}[theorem]{Proposition}

\newtheorem{corollary}[theorem]{Corollary}

\theoremstyle{definition}

\newtheorem{definition}[theorem]{Definition}

\newtheorem{claim}[theorem]{Claim}

\newtheorem{conjecture}[theorem]{Conjecture}

\newcommand{\mA}{\mathbb A}

\newcommand{\mC}{{\mathbb C}}

\newcommand{\mE}{{\mathbb E}}

\newcommand{\mF}{\mathbb F}

\newcommand{\mN}{\mathbb N}

\newcommand{\mP}{\mathbb P}

\newcommand{\mV}{\mathbb V}

\newcommand{\mX}{\mathbb X}

\newcommand{\mZ}{{\mathbb Z}}

\newcommand{\bo}{\omega}

\newcommand{\ep}{\varepsilon}

\newcommand{\kk}{\kappa}

\newcommand{\mcF}{\mathcal F}

\newcommand{\mcG}{\mathcal G}

\newcommand{\mcI}{\mathcal I}

\newcommand{\mcP}{\mathcal P}

\newcommand{\mcQ}{\mathcal Q}

\newcommand{\mcR}{\mathcal R}

\newcommand{\mcS}{\mathcal S}

\newcommand{\del}{\triangle}

\newcommand{\ti}{\tilde}

\author{Amichai Lampert and Tamar Ziegler}
\title{Relative rank and regularization}
\thanks{The authors are supported by ERC grant ErgComNum 682150, and ISF grant 2112/20.}
\begin{document}
	\maketitle

\begin{abstract} 
	We introduce a new concept of rank - {\em relative rank} associated to a filtered collection of polynomials. When the filtration is trivial our relative rank coincides with {\em Schmidt rank} (also called {\em strength}). We also introduce the notion of {\em relative bias}. The main result of the paper is  a relation between these two quantities over finite fields (as a special case we obtain a new proof of the results in \cite{Mi}). This relation allows us to get an accurate estimate for
 the number of points on an affine variety given by a collection of polynomials
which is high relative rank (Lemma \ref{Atom-size}). The key advantage of relative rank is that it allows one to perform an efficient regularization procedure which is {\em polynomial} in the initial number of polynomials (the regularization process with Schmidt rank is far worse than tower exponential). The main result allows us to replace Schmidt rank with relative rank in many key applications in combinatorics, algebraic geometry and algebra.  For example,  we prove that any collection of  polynomials $\mcP=(P_i)_{i=1}^c$  of degrees $\le d$ in a polynomial ring over an algebraically closed field of characteristic $>d$ is contained in an ideal $\mcI(\mcQ)$, generated by  a collection $\mcQ$ of  polynomials of degrees $\le d$ which form a regular sequence, and $\mcQ$ is of size $\le A c^{A}$, where $A=A(d)$ is independent of the number of variables.  
 \end{abstract}

\section{introduction}
  In \cite{S} Schmidt introduced a notion of complexity, the  {\em Schmidt rank} (also called the $h$-invariant)\footnote{This notion of complexity was reintroduced later in the work of Ananyan and Hochster \cite{ah} in their proof of Stillman's conjecture where it is called {\em strength}.} for a collection of polynomials $\mathcal P=(P_i)_{i=1}^c$, and used it to obtain conditions on the existence of integer solutions for a system of equations $\{P_i(x)=0\}_{i=1}^c$ defined over the rational numbers. The most natural polynomials on which to define this quantity are homogeneous ones. If $ P $ is a polynomial of degree $ d $ we denote by $ \tilde P $ the degree $ d $ homogeneous component of $ P. $ 		
	
\begin{definition}[Schmidt Rank]
	The {\em Schmidt rank} of a homogeneous polynomial  $P \in \mF[x_1, \ldots, x_n]$ is the minimal length of a presentation $P=\sum_{i=1}^r Q_iR_i,$ where for all $ i\in[r],\ Q_i, R_i $ are homogeneous polynomials of degree $ <\deg(P), $ and is denoted by $rk(P).$ For a general polynomial $ P $ we set $ rk(P) = rk(\tilde P). $ The rank of a collection of homogeneous polynomials $\mcP=(P_i)_{i=1}^c$ is 
	\[
	rk (\mcP) = \min \{ rk(a_1 P_1+\ldots+a_c P_c) :\ 0\neq a\in\mF^c \}.
	 \]
	If $ \mcP $ is a general collection of polynomials we set $ rk(\mcP) = rk(\tilde\mcP), $ where \\
	$ \tilde\mcP = (\tilde P_i)_{i=1}^c. $
\end{definition}	
	
	Now assume $\mF = \mF_q$ is a finite field and fix a non-trivial character $\chi:\mF\to\mC.$ We define the {\em bias} of a polynomial as follows.
	
	\begin{definition}
		For a function $P:\mF^n\to \mF$ we define 
		\[ 
		bias(P) = \left| \mE_{x\in \mF^n} \chi(P(x)) \right|,
		\]
		where $ \mE_{x\in E} := \frac{1}{|E|}\sum_{x\in E} $ for a finite set $ E, $ with the convention $ \mE_{x\in\emptyset} = 0. $ 
		This quantity may of course depend on $\chi,$ but we omit it from the notation.\footnote{For multilinear polynomials $bias(P)$ does not depend on the character $\chi$; in this case $-\log_q(bias(P))$ was introduced in \cite{gw} as the {\em analytic rank of $P$}.}
	\end{definition} 
	
	The relation between bias and rank can already be traced to the work of Schmidt; indeed, in his paper \cite{S} Schmidt showed that over the complex field the Schmidt rank of a polynomial is proportional to the codimension of the singular locus of the associated variety and used this for estimating related exponential sums. 
	
	In \cite{gt} Green and Tao observed that one can make the relation independent of the number of variables: \footnote{They use a slightly weaker definition of rank, but the results in their paper imply the theorem as stated.}
	
		\begin{theorem}[Bias implies low rank]\label{gt} Let $ \mF=\mF_q $ be a finite field, let  $0\le d < char(\mF)$, and let $s>0$.  There exists a constant $C=C(\mF,s,d)$ such that the following holds: For any $n$,  $ P\in\mF[x_1,\ldots,x_n]$ of degree $ \le d $ with $ bias(P) \ge q^{-s} $,  we have $rk(P) \le C$.
	\end{theorem}

	This result was subsequently extended to include the case $d \ge char(\mF) $ by Kauffman and Lovett in \cite{kl}. 
	Bhowmick and Lovett \cite{bl} showed that the bound can be made independent of the field $ \mF $ in the sense that we can take $ C = C(s,d). $ 
	Returning to the case $ d < char(\mF) $, effective bounds for $ C $ were given by
	Mili\'cevi\'c \cite{Mi} who proved 
	\[ 
	C = A(d) ( 1+s ) ^{B(d)}.
	 \] 
	Independently, Janzer \cite{janzer}
	proved a similar, slightly weaker result with $s \log q $ instead of $s$.  
	
	\begin{conjecture}\label{linear-br}   For all finite fields $\mF$ with  $char(\mF)>d$ we have \linebreak $C(\mF,s,d) = A(d)(1+ s)$.
	\end{conjecture}
	
	Conjecture \ref{linear-br} was recently proved in the case $d=3$ in \cite{akz}, \cite{CM} independently. \\

	A key feature of high rank collections of polynomials, that can be easily derived from Theorem \ref{gt}, is that all fibers of the map $\mcP:\mathbb \mF^n \to \mathbb \mF^c$ are essentially of the same size. 
\begin{theorem}[Size of fibers]\label{atoms}
Let $d>0$. Let $\mF=\mF_q$ be a finite field of characteristic $>d$. There exists $A= A(d)$ such that the following holds:  For any $n$, any collection of polynomials $\mcP=(P_i)_{i=1}^c$  in $\mF[x_1,\ldots,x_n]$ of degrees $\le d$ and rank $>Ac^A$,  any $a, b \in \mF^c$ we have $||\mcP^{-1}(a)| / |\mcP^{-1}(b)| - 1| <q^{-1}$.
\end{theorem}

Using model theoretic techniques one can translate the data in Theorem \ref{atoms} to an estimate  of the dimension of the fibers of the map $\mcP:\mathbb \mA^n \to \mathbb A^c$  for any algebraically closed field of characteristic zero or  $>\max \deg P_i$, see e.g. \cite{kz-survey,kz-extension}.  \\

In many proofs using the notion of Schmidt rank, a key procedure is regularization (e.g. \cite{S}, \cite{bl},  \cite{cm}): given a collection of polynomials, one applies procedure replacing the original collection of polynomials $\mcP$ with a new collection $\mcQ$ that is of high Schmidt rank compared to its size, (so Theorem \ref{atoms}  holds) and such that the ideal generated by $\mcP$ is contained in the 
ideal generated by $\mcQ$.  The drawback in this regularization procedures is that the size of $\mcQ$ is far worse than even tower exponential in the size of $\mcP$ (the bound that can be derived from Schmidts argument was worked out by Wooley in \cite{W}). \\

We introduce a new notion of rank {\em relative rank}, associated with a filtered collection of polynomials, both in the algebraic and analytic contexts. We replace the ambient space $ \mF^n $ by certain affine varieties and we replace Schmidt rank by an appropriate relative rank.  The definition of relative rank is somewhat technical; we defer it to the next section (Definition \ref{relative-rank}).
We remark that the new notion of relative rank coincided with the Schmidt rank and the analytic rank when the filtration consists of one set,  but is different otherwise.\\ 

Our main result (Theorem \ref{Bias-rank-poly}), described in detail in the next section, is a relation between relative rank and relative bias, similar to the relation described above between rank and bias. 
The key advantage in our new notion of rank, is that while it retains many of the properties of the ranks described above, for example Theorem \ref{atoms} remains valid, the  regularization procedure with respect to this notion is {\em polynomial} in the size of the original collection. As such we are able to give good quantitative bounds on a variety of problems in which a regularization procedure is used.  We provide an example of an algebraic application in Theorem \ref{reg-seq}.\\ 

In the interest of exposition, we will start by presenting a special case of our main result. This special case is actually the result we initially set out to prove. First we give our definitions of relative rank and relative bias.

	\begin{definition}[Relative rank]\label{relative-rank-int}
			 The relative rank of a homogeneous polynomial $P$ on a collection of homogeneous polynomials $ \mcQ = (Q_1,\ldots,Q_m) $ is
			\[
			rk_\mcQ (P) := \min \{rk(P+\sum_{i=1}^m R_i Q_i): \deg(R_i)+\deg(Q_i) \le \deg(P)\ \forall i\in[m] \}.
			\]
			Note that whenever $ \deg(Q_i) > \deg(P) $ this implies $ R_i = 0. $ For general polynomials, with $ \tilde P,\tilde \mcQ $ the corresponding homogeneous polynomial and collection,
			set $ rk_\mcQ (P) := rk_{\tilde\mcQ} (\tilde P). $
			\end{definition}

		\begin{definition}[Relative bias]\label{relative-bias}
			The relative bias of a function $P:\mF^n\to \mF$ on a subset $X\subset
			\mF^n$ is 
			\[ 
			bias_X (P) = \left| \mE_{x\in X} \chi(P(x)) \right|.
			\]
	\end{definition}

 	For a collection $ \mcQ = (Q_i)_{i\in[m]} $ of polynomials in $ \mF[x_1,\ldots,x_n] $, denote the zero locus by 
	\[ 
	Z(\mathcal{Q}) = \left\{ x\in\mF^n: Q_i(x) = 0\ \forall i\in[m] \right\}.
	 \]
	A special case of our main result is then:
	
	\begin{theorem}
		Let $ \mF $ be a finite field and $0\le d < char(\mF).$ There exist  constants $A(d), B(d)$ such that if $ \mcQ = (Q_i)_{i=1}^m $ is a collection of polynomials with degrees $ \le d $ with $ rk(\mcQ) > A(m+s)^B $ and $ P $ is a polynomial of degree $ \le d $ with\\
		 $ bias_{Z(\mcQ)}(P) \ge q^{-s} $ then we have 
		\[ 
		rk_\mcQ (P) \le A(1+s)^B.
		 \]
	\end{theorem}

	We briefly pause here and compare this to the result of Milicevic (and Janzer). For $ \mcQ = \emptyset $ this gives a new proof of their result. Moreover, when $ m \le s $ or when $ \mcQ $ is composed of polynomials of degree $ \ge \deg(P) $ the above theorem in fact follows from their result. But it is new in the case that $ \mcQ $ is composed of $ \gg s $ polynomials of lower degree than $ P. $ \\
	
	\emph{Acknowledgements.} The first author would like to thank the second author for introducing him to this question and its various applications. He would also like to thank his wife Noa for her constant support and encouragement. 
	

	\section{Main theorem}
	
	We now give our definitions in full and state our main theorem. 
	\begin{definition}\label{relative-rank}
		\begin{enumerate}
			\item The relative rank of a collection of homogeneous polynomials $\mcP = (P_i)_{i\in[m]}$ on another collection of polynomials $\mcQ$ is
			\[
			rk_\mcQ (\mcP) = \min \left\lbrace rk_\mcQ \left( a\cdot \mcP \right) :\ 0\neq a\in\mF^m \right\rbrace. 
			\]
			If $ \mcP $ is a general collection of polynomials with corresponding homogeneous collection $ \tilde\mcP, $ then $ 	rk_\mcQ (\mcP) = 	rk_\mcQ (\tilde \mcP). $
			\item A tower
			\[
			\mcQ = (Q_i)_{i\in [h]}
			\]
			of \textbf{height} $h$ is composed of $h$ collections of polynomials $\mcQ_i = (Q_{i,j}) _{j\in[m_i]}$ which we call \textbf{layers} such that the polynomials in each layer all have the same degree $d_i.$ The \textbf{degree} of the tower is $ \max \left(d_1,\ldots,d_h \right).$ The \textbf{dimension} of the tower is $ m_1+\ldots+m_h. $ We denote the truncated tower $\mcQ_{<i} = (\mcQ_j)_{j\in [i-1]}.$ 
			
			\item The tower $ \mcQ $ is $(A,B,s)$-regular if for every $i\in[h]$ we have  
			\[
			rk_{\mcQ_{<i}} (\mcQ_i) > A(m_i+m_{i+1}+\ldots+m_h+s)^B. 
			\]
		\end{enumerate}
	\end{definition}
		
	The main theorem we prove is the following: 
	\begin{theorem}[Relative bias implies relative low rank]\label{Bias-rank-poly}
		There exist  constants $A(d,H), B(d,H)$ such that if $\mcQ$ is an $(A,B,s)$-regular tower of degree $\le d < char(\mF)$ and height $\le H,$ and $P$ is a polynomial of degree $\le d$ with
		\[
			bias_{Z(\mcQ)}(P) \ge q^{-s},
		\]
		then
		\[
			rk_\mcQ (P) \le A(1+s)^B. 
		\] 
	\end{theorem}

	Theorem \ref{Bias-rank-poly} implies that relatively regular collections of polynomials enjoy similar equidistribution properties to those of high rank collections of polynomials. In fact, all the results in \cite{kz-extension, kz-survey, kz-uniform}, \cite{bl} which were proved for collections of polynomials of high rank, are also valid for relatively regular collections of polynomials. Some examples of results which transfer seamlessly to the relatively regular setting are given in the Appendix.

   The dependence of the constants $A,B$ on the parameters $d,H$ which can be extracted from our proof is quite weak and involves recursively defined functions. This is due to the rather involved inductive argument. It is natural to ask whether this result holds with a stronger dependence on $d,H.$

   \textbf{Question:} Can the constants $A,B$ in theorem \ref{Bias-rank-poly} be taken to be "reasonable" functions of $d,H$? Say, a tower of exponentials of bounded height as in Mili\'cevi\'c's result \cite{Mi}? 
				
	For a collection of polynomials $ \mcQ, $ we write $ I(\mcQ) $ for the ideal they generate. In many applications of the notion of rank (e.g \cite{S}, \cite{cm}) one is given a collection of polynomials, to which one applies a regularization process, replacing the original collection $\mcQ$ with a new regular collection $\mcQ'$, such that any $\mcQ \subset I(\mcQ').$ The issue is that the regularization process is very costly (see \cite{W}). One of the advantages of relative rank is that it allows a very efficient process of regularization.

	\begin{theorem}[Polynomial regularity]\label{polyreg}
			Let $A,B,d$ be given. There exist constants $ C(A,B,d),D(A,B,d) $
			such that if $\mcQ$ is a collection of $m$ homogeneous polynomials of positive degrees $\le d$, then there exists a polynomial tower $\mcQ'$  of dimension $\le C(m+s)^D,$ degree $ \le d $ and height $\le d$ such that $\mcQ'$ is
			$(A,B,s)$ regular, and $\mcQ\subset I(\mcQ').$  
	\end{theorem} 
		
		\begin{definition}
			A collection of polynomials $\mcQ = (Q_i)_{i=1}^m \subset \mF[x_1,\ldots,x_n]$ is {\em $s$-uniform} if 
			$$  \left| \frac{|Z(\mcQ)|}{q^{n-m}} -1 \right|  \le q^{-s}.$$
		\end{definition}
		As a corollary of Theorem \ref{Bias-rank-poly} we obtain:
		
		\begin{lemma}[Regular varieties are of the expected size]
			There exist constants $A(d,H),B(d,H)$ such that any $(A,B,s)$-regular polynomial tower $\mcQ$ of degree $\le d < char(\mF)$ and height $\le H$ is $s$-uniform.
		\end{lemma}

	In particular we obtain the following algebraic consequence: 
	\begin{theorem}\label{reg-seq}
		Let $d>0$. Let $\mF$ be an algebraically closed field of characteristic zero or $>d$. There exists a constant $A=A(d)$ such that if $\mcP$ is a collection of $c$ homogeneous polynomials of positive degrees $\le d,$ then there exists a collection $\mcQ$ of $ \le Ac^A$ homogeneous polynomials of degrees $\le d$ such that $\mcQ$ is a regular sequence, and $P_i \in \mcI(\mcQ)$.  
	\end{theorem}
	
	We prove Theorem \ref{reg-seq} in the Appendix; it follows from Theorem \ref{polyreg} and an adaptation of the results in \cite{kz-extension} to the context of relative rank.  \\
	
	We also obtain a robust Nullstellensatz result for regular collections.
	
	\begin{theorem}[Robust Nullstellensatz for regular collections]\label{null-poly}
		There exist constants $A(d,H),B(d,H),s(d)$ such that if $\mcQ = (Q_{i,j})_{i\in[h],j\in[m_i]} $ is an $(A,B,s)$-regular polynomial tower of deg $\le d < char(\mF)$ and height $ \le H, $ and $P$  is a polynomial of degree $ \le d $ which vanishes $ q^{-s} $-a.e. on $ Z(\mcQ), $ then there exist polynomials $ R_{i,j} $ satisfying $ \deg(R_{i,j}) + \deg(Q_{i,j}) \le \deg(P) $ and 
		\[ 
		P = \sum_{i,j} R_{i,j}Q_{i,j}.
		 \] 
	\end{theorem}

	In order to prove Theorem \ref{Bias-rank-poly}, as in \cite{Mi,janzer}, we first reduce it to an analogous statement for multilinear maps. However, the reduction of Theorem \ref{Bias-rank-poly} to its multi-linear version is not so straightforward in our case. This will be done by induction on the degree $ d $ and will require some intermediate results that appear during the proof of the theorem in the multi-linear case. Therefore, we postpone it to Section \ref{multilinear-reduction}.
	We now turn to the multi-linear case.

	\subsection{Multi-linear definitions and main theorem}
	Let $V_1,\ldots,V_k$ be finite-dimensional vector spaces over the finite field $\mF.$ For $I\subset [k],$ we denote \linebreak
    $V^I=\prod_{i\in I} V_i.$ For $ x\in V^{[k]}, $ we denote by $ x_I $ the projection of $ x $ to $ V^I. $
	
	In order to prove our main theorem for multi-linear functions, we will need to work with a slightly larger class of functions which we call full multi-affine maps. Recall that a multi-affine map $ P:V^I\to\mF $ has a unique presentation $ P(x_I) = \sum_{J\subset I} P_J(x_J) $ where $ P_J:V^J\to\mF $ is multi-linear for all $ J\subset I. $ By abuse of notation, we will also denote by $ P $ the function $ V^{[k]}\to\mF $ given by $ x\mapsto P(x_I). $
	
	\begin{definition}
		$P$ is called \textbf{full} if $P_I \neq 0.$ We denote by $\tilde P = P_I$ the multi-linear part of $P.$
	\end{definition}

	In \cite{neslund}  a notion of rank for multi-linear functions is defined, called {\em partition rank}.
	\begin{definition}
		Let $ P:V^I\to\mF $ be a multi-linear function. The \textbf{partition rank} of $ P $ is the minimal length of a presentation $ P(x) = \sum_{i=1}^r Q_i(x_{I_i})R_i(x_{I\setminus I_i}), $ where $ \emptyset \neq I_i \subsetneq I $  and $ Q_i:V^{I_i}\to\mF,\ R_i:V^{I\setminus I_i}\to\mF $ are multi-linear for all $ i\in[r]. $ We write $ r = prk(P). $ We extend this definition to full multi-affine maps via $ prk(Q) = prk(\tilde Q). $
	\end{definition} 

	The bias of $ P $ on $ V^{[k]} $ is defined as before, but it has some interesting properties when $ P:V^I\to\mF $ is multi-linear. Writing $ P(x) = A(x_{I\setminus \{i_0\}})\cdot x_{i_0} $ where $ A:V^{I\setminus \{i_0\}}\to V_{i_0} $ is multi-linear for some $ i_0\in[k], $ Fourier analysis yields 
	\[ 
	bias(P) = |\mE_{x\in V^{[k]}} \chi (P(x))| = \mP_{x\in V^{I\setminus \{i_0\}}} \left( A(x) = 0 \right),  
	 \]
	where $ \mP_{x\in E} (T):= \frac{|T|}{|E|} $ for a subset $ T\subset E. $
	In particular, the character sum above is always positive and does not depend on our choice of character $ \chi. $ it is not too difficult to show that multi-linear maps with low partition rank exhibit significant bias, see \cite{kz-approx, lovett-rank}.
	
	\begin{claim}[Low partition rank implies bias] \label{rank-bias}
		If $ P:V^I\to\mF $ is multi-linear with $ prk(P) =r $ then 
		\[ 
		\mE_{x\in V^{[k]}}\chi(P(x)) \ge q^{-r} .
		 \] 
	\end{claim}
	
	In \cite{Mi} the following (much more difficult) inequality was proved in the other direction:
	
	\begin{theorem}[Bias implies low partition rank] 
		Let $ \mF $ be a finite field and $0\le d < char(\mF).$ Suppose $ P:V^I\to\mF$ is multi-linear with $ |I|\le d $ and $ bias(P) \ge q^{-s}.$  Then 
		\[
		prk(P) \le A(1+s)^B,
		\] 
		where $ A=A(d),B=B(d) $ are constants.
	\end{theorem}

	This result immediately extends to full multi-affine maps because a short calculation in \cite{Lov} (which we will follow in Lemma \ref{affine-to-linear}) shows that
	\[ 
	bias(P) \le bias(\tilde P),
	 \] 
	and by definition they have the same partition rank. We now make the corresponding relative definitions for full multi-affine maps.
	
	\begin{definition}
		\begin{enumerate}
			\item 	Let $ \mcQ=(Q_i:V^{I_i}\to\mF)_{i\in[m]} $ and $ P:V^I\to\mF $ be multi-linear maps. The \textbf{partition rank of $ P $ relative to $ \mcQ $} is 
			\[ 
			prk_\mcQ (P) = \min \left\lbrace prk\left( P-\sum_{i\in[m],I_i\subset I} Q_i(x_{I_i})R_i(x_{I\setminus I_i}) \right) :\ R_i\ \text{are multi-linear}\right\rbrace. 
			\]
			If $ P,\mcQ $ are full multi-affine, we define $ prk_\mcQ (P) = prk_{\tilde\mcQ} (\tilde P). $
			
			\item The partition rank of a collection of multi-linear maps \linebreak
            $ \mcP = (P_i:V^I\to\mF)_{i\in[c]} $ relative to $ \mcQ $ is 
			\[ 
			prk_\mcQ (\mcP) = \min \left\lbrace prk_\mcQ \left( a\cdot\mcP \right):\ 0\neq\textbf{a}\in\mF^c  \right\rbrace. 
			 \]  
			 If $ \mcP,\mcQ $ are full multi-affine with all maps in $ \mcP $ defined on a fixed index set $ I, $ set $ prk_\mcQ (\mcP) = prk_{\tilde \mcQ} (\tilde \mcP). $
			
			\item A tower  
			\[
			\mcQ = (\mcQ_i)_{i\in [h]}
			\]
			is composed of $h$ collections of full multi-affine functions \linebreak
            $ \mcQ_i = (Q_{i,j}:V^{I_i}\to\mF)_{j\in[m_i]}, $ which we call \textbf{layers}. The \textbf{associated multi-linear tower} is $\tilde \mcQ = (\tilde \mcQ_i)_{i\in[h]}.$ We say the tower is multi-linear if $\tilde\mcQ = \mcQ.$ 
		The \textbf{degree} of the tower is $\max_i |I_i|,$ and the \textbf{height} of the tower is
			$h.$ The \textbf{dimension} of the tower is  $m = m_1+\ldots+m_h.$ For $i\in[h],$ we denote the truncated tower by $ \mcQ_{<i} = (\mcQ_j)_{j\in [i-1]}. $
			
			\item Let $A,B,s>0.$ We say that $\mcQ$ is $(A,B,s)$-regular if for each $i\in [h]$ we have
			\[
			prk_{\mcQ_{<i}}(\mcQ_i) > A\left(m_i+m_{i+1}+\ldots+m_h+s\right)^B.
			\]
		\end{enumerate}
	\end{definition}

	Our main theorem for full multi-affine maps is
	
	\begin{theorem}[Relative bias implies low relative rank (d,H)]\label{Bias-Rank}
		There exist constants $A(d,H),B(d,H)$ such that if $ P:V^I\to\mF$  is a full multi-affine map with $|I|\le d$, $\mcQ$ is a tower on $V^{[k]}$ of degree $\le d$ and height $\le H$  such that:
		\begin{enumerate}
			\item $bias_{Z(\mcQ)} (P) \ge q^{-s}$ and
			\item $\mcQ$ is $(A,B,s)$-regular,
		\end{enumerate}
		then $rk_\mcQ (P) \le A(1+s)^B.$
	\end{theorem}

	We will also prove a relative version of Claim \ref{rank-bias}.  
	
	\begin{claim}[Low relative rank implies relative bias]
		There exist constants $A(d,H), B(d,H)$ such that if $ \mcQ $ is an $(A,B,s)$-regular multi-linear tower on $V^{[k]}$ of degree $\le d$ and height $\le H$ and $P:V^I\to\mF$ with $|I| \le d$ is multi-linear with $rk_\mcQ (P) \le r,$ then
		\begin{enumerate}
			\item $Re[\mE_{x\in Z(\mcQ)} \chi(P(x))] \ge q^{-r}-q^{-s}$ and
			\item $|Im[\mE_{x\in Z(\mcQ)} \chi(P(x))]| \le q^{-s}.$
		\end{enumerate}
	\end{claim}

	Let us now give a high-level sketch of the proof of Theorem \ref{Bias-Rank}. This will proceed by induction on $d,$ where the base case $ d = 1 $ follows from basic Fourier analysis. Assuming the theorem holds at degree $ d, $ we present some useful consequences in section \ref{consequences} and make some reductions in section \ref{reductions}. After these reductions, the task at hand is to prove the theorem for multi-linear $ P:V^{[d+1]}\to\mF $ and a multi-linear tower $ \mcQ $ on $ V^{[d+1]} $ of degree $ \le d. $ In section \ref{approximations}, we show that there are many "derived" towers $ \mcQ_t $ and multi-linear functions $ \phi_t:V^{[d]}\to\mF^{100s} $ such that $ P $ vanishes $ q^{-s} $-a.e. on the variety $ Z(\mcQ_t)\cap \{\phi_t(x_{[d]}) = 0\}. $ This is similar in spirit to previous work on this subject, but the derived towers appearing here present a new challenge. In section \ref{regularization-sec} we pause to describe the relative regularization process which is the multi-linear analogue of Theorem \ref{polyreg}. In section \ref{fixing}, the regularization process is applied to the collections $ \mcQ_t,\phi_t $ to obtain regular varieties on which $ P $ vanishes $ q^{-s} $-a.e. In that section we also show that this means that $ P $ vanishes identically on them, similarly to proposition 5.1 in \cite{gt}. A key novel step in our proof consists in then showing that these varieties can be "summed" to obtain a single tower $ \mcR $ of dimension $ \le A(1+s)^B $ such that $ P\restriction_{Z(\mcQ\cup\mcR)} \equiv 0$ and $ \mcQ\cup \mcR $ is regular. In section \ref{nullstellensatz-sec}, we prove a nullstellensatz which allows us to conclude that $ rk_{\mcQ\cup\mcR} (P) = 0, $ so $ rk_\mcQ (P) \le \dim(\mcR) \le A(1+s)^B,$ which completes the inductive step.  
	
	\section{Some equidistribution lemmas}\label{consequences}
	
	As mentioned before, the proof of Theorem \ref{Bias-Rank} will proceed by induction on $ d $ and the base case $ d=1 $ follows from Fourier analysis. For the purposes of the inductive argument it will be useful to formulate a version of Theorem \ref{Bias-Rank} with an auxiliary parameter, which we now do.
	
	\begin{theorem}[Relative bias implies low relative rank (d,H,l)]\label{Bias-Rank-par}
	There exist constants $A(d,H,l),B(d,H,l)$ such that if $ P:V^I\to\mF$ is a full multi-affine map with $|I|\le d$, $\mcQ$ is a tower on $V^{[k]}$ of degree $\le d$ and height $\le H$ such that:
	\begin{enumerate}
		\item $bias_{Z(\mcQ)} (P) \ge q^{-s},$
		\item $\mcQ$ is $(A,B,s)$-regular and
		\item $ \mcQ_{\le H-l} $ is of degree $ <d,$
	\end{enumerate}
	then $rk_\mcQ (P) \le A(1+s)^B.$
	\end{theorem}

	Theorem \ref{Bias-Rank} corresponds of course to the case $ l=H. $ We call $ \mcQ $ as above a tower of type  (d,H,l) for short.
	In this section we assume Theorem \ref{Bias-Rank-par}(d,H,l) holds and explore some of its useful consequences. We warm up with the following simple but important result.
	
	\begin{lemma}[Regular varieties are of the expected size (d,H+1,l+1)]\label{Atom-size}
		There exist constants $A(d,H,l),B(d,H,l)$ such that for any $s>0$ and for any $(A,B,s)$-regular tower $\mcQ$ on $V^{[k]}$ of type (d,h,l+1)  and dimension $m,$ we have 
		\[
		\left|q^m \frac{|Z(\mcQ)|}{|V^{[k]}|}-1\right|\le q^{-s}.
		\]
	\end{lemma}
	
	\begin{proof}
		Let $A',B'$ be the constants from Theorem \ref{Bias-Rank-par}(d,H,l). If $A, B$ are sufficiently large depending on $ A',B' $ then $\mcQ$ is $(A',B',s+H+2)$ regular. So for any $i\in[H+1]$ the tower $\mcQ_{<i}$ is $(A',B',s+H+m_i)$-regular of type (d,H,l)
		and
		\[
		rk_{\mcQ_{<i}}(\mcQ_i) > A'(s+H+m_i+2)^{B'}.
		\]
		By Theorem \ref{Bias-Rank-par} it follows that for any $0\neq a\in \mF^{m_i}$ we have
		\[
		bias_{Z(\mcQ_{<i})} (a\cdot \mcQ_i) < q^{-(s+H+m_i+1)}.
		\]
		Fourier analysis yields
		\[
		 q^{m_i}\frac{|Z(\mcQ_{\le i})|}{|Z(\mcQ_{<i})|} = \sum_{a\in \mF^{m_i}}\mE_{x\in Z(\mcQ_{<i})} \chi( a\cdot \mcQ_i(x)) = 1+e_i,
		 \]
		where $|e_i|\le q^{-(s+H+1)}.$
		Then we multiply to get
		\[
		\frac{|Z(\mcQ)|}{|V^{[k]}|} = \prod_{i\in[H+1]}\frac{|Z(\mcQ_{\le i})|}{|Z(\mcQ_{<i})|} = \prod_{i\in[H+1]}  (1+e_i) q^{-m_i} = (1+e)q^{-m},
		\]
		where $|e|\le q^{-s}.$
	\end{proof}
	
	We now introduce some notation. For a tower $\mcQ$ on $V^{[k]}$ and a subset $I\subset [k],$ we write $ \mcQ_I = (\mcQ_i)_{i\in[h],I_i\subset I} $ for the restriction to a tower on $ V^I. $ Similarly, we define $ Z(\mcQ)_I = Z(\mcQ_I) \subset V^I. $  For $y\in V^I$ and a full multi-affine function $ Q:V^J\to\mF, $ the full multi-affine function $Q(y):V^{J\setminus I}\to\mF $ is defined by $ Q(y)(z)=Q(y_{I\cap J},z). $ set $ \mcQ(y) = (\mcQ_i(y))_{i\in[h]} $ for the tower induced on $ V^{[k]\setminus I}, $ where $ \mcQ_i(y) = (Q_{i,j}(y):V^{I_i\setminus I}\to\mF)_{j\in[m_i]}. $ For the zero locus we write $ Z(\mcQ) (y) = Z(\mcQ(y)). $ 
	
	We now state several interrelated lemmas.
	
	\begin{lemma}[Derivatives(d,H+1,l+1)]\label{Derivatives}
		For any $C,D$ there exist constants $A(C,D,d,H,l),$ $B(C,D,d,H,l)$ such that if $\mcQ$ is an $(A,B,s)$-regular tower on $V^{[k]}$ of type (d,H+1,l+1) and $I\subset [k],$ then for $q^{-s}$-a.e. $y\in Z(\mcQ)_I$ the tower $\mcQ(y)$ on $ V^{[k]\setminus I} $ is $(C,D,s)$-regular. 
	\end{lemma}
	
	\begin{lemma}[Fubini(d,H+1,l+1)]\label{Fubini}
		There exist constants $A(d,H,l),B(d,H,l)$ such that if $\mcQ$ is an $(A,B,s)$-regular tower on $ V^{[k]} $ of type (d,H+1,l+1), $ I\subset[k], $ and $g:V^{[k]}\to\mC$ satisfies $ |g(x)|\le 1 $ for all $ x\in V^{[k]} $, then
		\[
		\left| \mE_{x\in Z(\mcQ)}g(x) - \mE_{y\in Z(\mcQ)_I}\mE_{z\in Z(\mcQ)(y)}g(y,z) \right| \le q^{-s}.
		\]
	\end{lemma}
	
	\begin{lemma}[Rank-bias(d,H+1,l+1)]\label{Rank-Bias}
		There exist constants $A(d,H,l), B(d,H,l)$ such that if $ \mcQ $ is an $(A,B,s)$-regular multi-linear tower on $V^{[k]}$ of type (d,H+1,l+1) and $P:V^I\to\mF$ is a multi-linear map with $|I| \le d+1$ then
		\begin{enumerate}
			\item $|Im[\mE_{x\in Z(\mcQ)} \chi(P(x))]| \le q^{-s}$ and
			\item if $rk_\mcQ (P) = r,$ then
			\[ 
			Re[\mE_{x\in Z(\mcQ)} \chi(P(x))] \ge q^{-r}-q^{-s}.
			 \]	
		\end{enumerate}
	\end{lemma}
	
	The above lemmas will be proved together by induction on the height. The implications, assuming theorem \ref{Bias-Rank-par}(d,H,l) throughout, are as follows
	\[
	\text{Rank-bias(d,H,l+1) + Fubini(d,H,l+1)}\implies \text{Derivatives(d,H+1,l+1)}
	\]
	\[
	\text{Derivatives(d,H+1,l+1)}\implies \text{Fubini(d,H+1,l+1)}
	\]
	\[
	\text{Derivatives(d,H+1,l+1) + Fubini(d,H+1,l+1)}\implies \text{Rank-bias(d,H+1,l+1)}
	\]
	
	In the base case of height $0,$ where the tower is empty, Lemmas \ref{Derivatives} and \ref{Fubini} are trivial and Lemma $\ref{Rank-Bias}$ is proved in \cite{kz-approx, lovett-rank}. We assume in the proofs that the tower is indeed of height $H+1,$ otherwise we are done by the inductive hypothesis.
	
	\begin{proof}[Proof of Lemma \ref{Derivatives} (derivatives(d,H+1,l+1))]
		We are interested in the regularity of $\mcQ(y),$ which is by definition the regularity of its associated multi-linear tower. Which multilinear maps appear? For $ Q_{i,j}\in\mcQ_i ,$ write $Q_{i,j} = \sum_{J\subset I_i} Q_{i,j}^J$ as a sum of multi-linear maps $ Q_{i,j}^J:V^J\to\mF, $ and set $Q'_{i,j} = \sum_{I_i\setminus I\subset J\subset I_i} Q_{i,j}^J,$ $\mcQ'_i = \{Q'_{i,j}\}_{j\in m_i}.$ Note that the following properties hold:
		\begin{enumerate}
			\item $\mcQ,\mcQ'$ have the same associated multi-linear variety (and hence the same regularity),
			\item $Z(\mcQ')_I = Z(\mcQ)_I$ and
			\item for any $y\in Z(\mcQ)_I,$ the multi-linear tower associated with $\mcQ(y)$ is $\mcQ'(y).$
		\end{enumerate}
		So after replacing $\mcQ$ with $\mcQ'$, we can assume without loss of generality that $\mcQ(y)$ is a multi-linear tower for any $y\in Z(\mcQ_I).$ Write $X=Z(\mcQ)$ for short. Let $ A',B' $ be the constants of Lemma \ref{Derivatives} at height $H.$ We deal separately with two cases.
		
		\textbf{The first case} (and the easier one) is when $I_{H+1}\subset I.$ In this case for any $y\in (X_I)_{<H+1}$ we have $\mcQ(y) = \mcQ_{<H+1}(y).$  For $A = 2^{B'}A',$ $B=B',$ we get that $\mcQ_{<H+1}$ is $(A',B',s+m_{H+1}+1)$-regular. Then by the inductive hypothesis we have that for $q^{-(s+m_{H+1}+1)}$ a.e. $y\in (X_I)_{<H+1}$ the tower $\mcQ(y)=\mcQ_{<H+1}(y)$ is $(C,D,s)$-regular. By Lemma \ref{Atom-size} we know that if $A,B$ are sufficiently large then  $\frac{|(X_I)_{<H+1}|}{|X_I|} \le q^{m_{H+1}+1}.$ Therefore $\mcQ(y)$ is regular for $q^{-s}$-a.e. $y\in X_I,$ as desired. 
		
		\textbf{The second case} is when $I_{H+1} \not\subset I.$ In this case we have $X_I = (X_I)_{<H+1}.$ If $A,B$ are sufficiently large then $\mcQ$ is $(A',B',100C(m_{H+1}+s)^D)$-regular. Let 
		\[
		s' = 100C(m_{H+1}+s)^D\ ,
		\  \varepsilon = q^{-s'}.
		\]
		For the rest of this proof, we will assume wherever needed that $ A,B,A',B',C,D $ are sufficiently large with respect to $ d,h,l.$
		By the inductive hpothesis we have that for $\varepsilon$-a.e. $y\in X_I$ the tower $\mcQ_{<H+1}(y)$ is $(C,D,s')$-regular. By a union bound, it is enough to show that for $ q^{-(s+1)} $-a.e. $y\in X_I$ we have
		\[
		 rk_{\mcQ_{<H+1}(y)} \left( \mcQ_{H+1}(y) \right) > C(m_{H+1}+s)^D. 
		\]
		 Fix $0\neq a\in\mF^{m_{H+1}}.$  Applying Theorem \ref{Bias-Rank-par}, we have
		\[
		Re \left[ \mE_{x\in X_{<H+1}} \chi(a\cdot \mcQ_{H+1}(x)) \right] \le \varepsilon.
		\]
		Lemma \ref{Fubini} at height $H$ yields $ \mE_{y\in X_I}  Re \left[ \mE_{z\in X_{<H+1}(y)}\chi(Q(y,z)) \right]  \le 2\varepsilon. $
		For $ \ep $-a.e. $ y\in X_I $ we can apply Lemma \ref{Rank-Bias} at height $H$ to get 
		\[ 
		 \left| \mE_{z\in X_{<H+1}(y)}\chi(Q(y,z)) \right| \le  Re \left[ \mE_{z\in X_{<H+1}(y)}\chi(Q(y,z)) \right] + \ep.
		 \]
		Therefore,   
		\[ 
		 \mE_{y\in X_I} |\mE_{z\in X_{<H+1}(y)}\chi(Q(y,z))| < 4\ep.
		\]
		By Markov's inequality, this means that for $2\sqrt{\ep}$-a.e. $y\in X_I,$ we have 
		\[ 
		|\mE_{z\in X_{<H+1}(y)}\chi(Q(y,z))| < 2\sqrt{\ep}.
		\]
		Then for $3\sqrt{\varepsilon}$-a.e. $y\in X_I$ both the above inequality holds and the tower $\mcQ^{<H+1}(y)$ is $(C,D,s')$-regular. Let $r_a(y) = rk_{\mcQ^{<H+1}(y)} (a\cdot \mcQ_{H+1}(y)).$ Applying Lemma \ref{Rank-Bias} at height $H$ we find that for these $y$ we have 
		\[
		q^{-r(y)} < 4\sqrt{\ep} < q^{-C(m_{H+1}+s)^D},
		\]
		which implies $r_a(y)>C(m_{H+1}+s)^D.$ 
		Running over all $0\neq a\in\mF^{m_{H+1}},$ we get that for $3\sqrt{\ep}q^{m_{H+1}}\le q^{-(s+1)}$-a.e. $y\in X_I$ we have $ r(y) > C(m_{H+1}+s)^D$ where $ r(y) = \max_{0\neq a\in\mF^{m_{H+1}}} r_a(y). $ This completes the proof. 
	\end{proof}
	
	\begin{proof}[Proof of Lemma \ref{Fubini} fubini(d,H+1,l+1)] By the triangle inequality,
		\[
		\left| \mE_{x\in X}g(x) - \mE_{y\in X _I}\mE_{z\in X(y)}g(y,z) \right| \le \mE_{y\in X _I} \left| \frac{|X _I|\cdot |X(y)|}{|X|} - 1 \right| . 
		 \] 
		Denoting $ f(y) = \frac{|X _I|\cdot |X(y)|}{|X|},\ \varepsilon = q^{-(s+2)}, $ and assuming $ A,B $ are sufficiently large, we note the following:
		\begin{enumerate}
			\item $ f(y) \ge 0$ for all $y\in X _I, $
			\item $ \mE_{y\in X _I} f(y) = 1 $ and
			\item by Lemmas \ref{Derivatives}(d,H+1,l+1) and \ref{Atom-size}(d,H+1,l+1), for $\ep$-a.e. $ y\in X _I, $ we have $ 1-\ep < f(y) < 1+\ep.$
		\end{enumerate}
		Let $ E\subset  X_I $  be the set of $y$ for which this last inequality holds, and let $ \rho = \frac{|E|}{|X _I|} $ (so $ 1-\ep\le\rho\le 1 $).  Then
		\begin{equation}\label{Fub-Ineq}
			\mE_{y\in X _I} |f(y)-1| = \rho\cdot\mE_{y\in E} |f(y)-1| + (1-\rho) \mE_{y\in E^c} |f(y)-1|. 
		\end{equation} 
		The first term is $\le\ep $ by the definition of $ E. $ For the second term we have
		\[
		(1-\rho) \mE_{y\in E^c} |f(y)-1| \le \ep + (1-\rho) \mE_{y\in E^c} f(y).
		  \]
		By the three properties listed above, we get that $ (1-\rho) \mE_{y\in E^c} f(y) \le 2\ep. $
		Plugging this into inequality \eqref{Fub-Ineq} yields
		\[
		 	\mE_{y\in X _I} |f(y)-1| \le 4\ep \le q^{-s}.
		 \] 
	\end{proof}
	
	\begin{proof}[Proof of Lemma \ref{Rank-Bias} rank-bias(d,H+1,l+1)]
		Let $ X = Z(\mcQ). $ First we apply Lemma \ref{Fubini}(d,H+1,l+1) to get 
		\begin{enumerate}
			\item $Re[\mE_{x\in X} \chi(P(x))] \ge Re[\mE_{x\in X_I} \chi(P(x))]-q^{-2s} $ and
			\item $|Im[\mE_{x\in X} \chi(P(x))]| \le |Im[\mE_{x\in X_I} \chi(P(x))]|+ q^{-2s}.$
		\end{enumerate}
		so we can assume without loss of generality that $I=[k].$ We will prove the lemma by induction on $|I|.$ For $|I| = 1$ the lemma follows from basic Fourier analysis. For the general case,
		we start by showing the second inequality holds. Apply Lemma \ref{Fubini} again to get 
		\[ 
		|Im[\mE_{x\in X} \chi(P(x))]| \le \mE_{y\in X_1}|Im[\mE_{z\in X(y)} \chi(P(y,z))]|+q^{-10s}.
		\]
		Let $A',B'$ be the constants for Lemma \ref{Rank-Bias} on an index set of size $ |I| -1. $ By Lemma \ref{Derivatives}(d,H+1,l+1), if $A,B$ are sufficiently large then for 
		$q^{-10s}$-a.e. $y\in X_1$ the multi-linear tower $\mcQ(y)$ is $(A',B',10s)$-regular. The induction hypothesis implies that
		\[ 
		\mE_{y\in X_1}|Im[\mE_{z\in X(y)} \chi(P(y,z))]| \le q^{-5s},
		\] 
		as desired. 
		
		Now for the first inequality, suppose modulo the maps in $ \mcQ $  we have the low partition rank presentation
		\[
		 P(x_{[k]}) = \sum_{i=1}^m l_i(x_1)g_i(x_{[2,k]}) + \sum_{j=m+1}^r q_i(x_{I_j})r_i(x_{[k]\setminus I_j}),
		 \]
		where $I_j\neq \{1\},[2,k].$
		Let $W = \{y\in X_1:l_i(y) = 0\ \forall i\in[m]\}.$ By Lemma \ref{Fubini},
		\begin{multline*}
			Re[\mE_{x\in X} \chi(P(x))] \ge \mE_{y\in X_1}Re[\mE_{z\in X(y)} \chi(P(y,z))]-q^{-10s} \\
			= \frac{|W|}{|X_1|}\mE_{y\in W}Re[\mE_{z\in X(y)} \chi(P(y,z))] + (1-\frac{|W|}{|X_1|})\mE_{y\in X_1\setminus W}Re[\mE_{z\in X(y)} \chi(P(y,z))]-q^{-10s}.
		\end{multline*}
		Note that:
		\begin{enumerate}
			\item for any $y\in W$ the rank of $P(y,\cdot)$ on $X(y)$ is $\le r-m,$
			\item $|W|/|X_1| \ge q^{-m}$ and
			\item for 
			$q^{-10s}$-a.e. $y\in X_1$ the multi-linear tower $\mcQ(y)$ is $(A',B',10s)$-regular.
		\end{enumerate}
		By these properties and the induction hypothesis for $ P(y,\cdot) $ on $ X(y) $ we can lower bound the first summand by
		\[
		 \frac{|W|}{|X_1|}\mE_{y\in W}Re[\mE_{z\in X(y)} \chi(P(y,z))] \ge q^{-r}-q^{-5s},
		\]
		and the second summand by
		\[
		 (1-\frac{|W|}{|X_1|})\mE_{y\in X_1\setminus W}Re[\mE_{z\in X(y)} \chi(P(y,z))] \ge -q^{-5s}.
		\]
	\end{proof}
	
	\section{Some reductions}\label{reductions}
	In this section we reduce the proof of Theorem \ref{Bias-Rank} for degree $d+1$ to the $ l=0 $ case of Theorem \ref{Bias-Rank-par}.
	
	\begin{lemma}
		Suppose Theorem \ref{Bias-Rank-par}(d+1,H,l) holds for $ l=0 $ and all $ H $. Then it holds for all $ l,H. $ 
	\end{lemma}

	\begin{proof}
		By induction on $ l. $ Suppose the theorem holds for (d+1,H,l). Let $ \mcQ $ be an $(A,B,s)$-regular tower on $ V^{[k]} $ of type (d+1,H,l+1) and $ P:V^I\to\mF $ a full multi-affine map of degree $ \le d+1, $ with $ bias_{Z(\mcQ)} (P) \ge q^{-s}. $ Again, we denote $ X=Z(\mcQ) $ for short. We saw in the last section that Theorem \ref{Bias-Rank-par}(d+1,H,l) implies Lemma \ref{Fubini} for (d+1,H+1,l+1), so
		\[ 
		| \mE_{x\in X_I} \chi(P(x)) | \ge q^{-2s},
		 \]
		i.e. we can assume without loss of generality that $ I=[k]. $ If $ \mcQ $ is of degree $ \le d $ then we are done. Suppose not, then $ k=d+1 $ and there is some layer $ t\in[H] $ with $ \mcQ_t $ of degree $ d+1 $. Let $ \mcQ_{\neq t} = (\mcQ_i)_{i\in[H]\setminus\{t\}}$ and $ X_{\neq t} = Z(\mcQ_{\neq t}). $ By Fourier analysis, we have 
		\begin{equation}\label{layer-peeling}
			|\mE_{x\in X} \chi(P(x))| \le  \frac{|X_{\neq t}|}{|X|}q^{-m_t} \sum_{a\in\mF^{m_t}} | \mE_{x\in X_{\neq t}}\chi(P(x)+a\cdot \mcQ_t(x)) |.
		\end{equation} 
		
		By Lemma \ref{Atom-size}(d+1,H,l+1) (which follows from Theorem \ref{Bias-Rank-par}(d+1,H,l)),\\
		 $ \frac{|X_{\neq t}|}{|X|}q^{-m_t} \le q. $ As for the second term, choose  $ b\in\mF^{m_t} $ such that \linebreak 
        $ prk_{\mcQ_{\neq t}} ( P(x)+b\cdot \mcQ_t(x) ) $ is minimal. By sub-additivity of relative partition rank, for any $ b\neq a \in\mF^{m_t}$ we must have
		 \[ 
		 prk_{\mcQ_{\neq t}} ( P(x)+a\cdot \mcQ_t(x) ) \ge \frac{A}{2}(m_t+s)^B.
		  \]
		By Theorem \ref{Bias-Rank-par}(d+1,H,l) (assuming $ A,B $ are sufficiently large) this means that for any $  b\neq a \in\mF^{m_t}, $ 
		\[ 
		 | \mE_{x\in X_{\neq t}}\chi(P(x)+a\cdot \mcQ_t(x)) | \le q^{-m_t-4s}.
		 \] 
		Plugging these estimates into inequality \eqref{layer-peeling} yields 
		\[ 
		q^{-2s-1} \le  | \mE_{x\in X_{\neq t}}\chi(P(x)+b\cdot \mcQ_t(x)) |.
		 \]
		Now apply Theorem \ref{Bias-Rank-par}(d+1,h,l) to conclude that 
		\[ 
		prk_\mcQ (P(x)) \le prk_{\mcQ_{\neq t}} ( P(x)+b\cdot \mcQ_t(x) ) \le A(1+s)^B.
		 \] 
	\end{proof}

	Therefore all we have to do is prove Theorem \ref{Bias-Rank-par}(d+1,H,0) in order to deduce Theorem \ref{Bias-Rank} at degree $ d+1. $ The following lemma allows us to reduce further to the case that $ P,\mcQ $ are multi-linear and defined on the same index set. From now on we assume Theorem \ref{Bias-Rank}(d,H) holds for all $ H, $ and work towards proving Theorem \ref{Bias-Rank-par}(d+1,H,0) for all $ H.$ We may safely forget about the auxillary parameter $ l $ of Theorem \ref{Bias-Rank-par}, it has served its purpose. The following lemma allows us to reduce further to the case where $ P,\mcQ $ are multi-linear and are defined on the same index set.
	\begin{lemma}\label{affine-to-linear}
		There exist $A(d,h),B(d,h)$ such that if $ \mcQ $ is an $ (A,B,s) $-regular tower of degree $ \le d $ and height $ \le h $ on $ V^{[k]}$ and $P:V^I\to\mF$ is a full multi-affine map of degree $|I|\le d+1$ then 
		\[
		 |\mE_{x\in X} \chi(P(x))| \le Re[\mE_{x\in \tilde{X}_I} \chi(\tilde{P}(x))] + q^{-s},
		 \]
		where $ X=Z(\mcQ) $ and $ \tilde P,\tilde X $ are the corresponding multi-linear map and variety.
	\end{lemma}
	
	\begin{proof}
		By applying Lemma \ref{Fubini}, we get 
		\[
		 |\mE_{x\in X} \chi(P(x))| \le |\mE_{x\in X_I} \chi(P(x))| + q^{-2s},
		 \]
		so we can assume without loss of generality that $ I=[k]. $ Now choose some $i\in [k]$ and write $ P(x) = A(x_{[k]\setminus\{i\}})\cdot x_i + b(x_{I\setminus\{i\}}),$ where  $A,b$ are multi-affine maps.
		Again using Lemma \ref{Fubini} twice, we have 
		\begin{multline*}
			|\mE_{x\in X} \chi(P(x))| \le |\mE_{x\in X_{I\setminus\{i\}}} \chi(b(x)) \mE_{y\in X(x)} \chi(A(x)\cdot y)| + q^{-10ks} \\
			\le \mE_{x\in X_{I\setminus\{i\}}}  |\mE_{y\in X(x)} \chi(A(x)\cdot y)| + q^{-10ks} \\
			=  \mE_{x\in X_{I\setminus\{i\}}} \mE_{y\in \widetilde{X(x)}} \chi(A(x)\cdot y) + q^{-10ks} \\
			\le |\mE_{x\in X'} \chi(A(x_{[k]\setminus\{i\}})\cdot x_i)|+ q^{-5ks},
		\end{multline*}
		where $X'$ is the variety associated with $ X $ which is linear in the $i$-th coordinate and $ A(x_{[k]\setminus\{i\}})\cdot x_i $ is the multi-affine map associated with $ P $ which is linear in the $i$-th coordinate. Iterating this inequality over and over for all the other coordinates, we end up with
		\begin{multline*}
			|\mE_{x\in X} \chi(P(x))| \le |\mE_{x\in \tilde X} \chi(\tilde P(x))|+(k+1)q^{-5ks}
			\le Re[\mE_{x\in \tilde X} \chi(\tilde f(x))] + q^{-s},
		\end{multline*}
		where the second inequality follows from Lemma \ref{Rank-Bias}.
	\end{proof}

	\section{Many low rank approximations}\label{approximations}

	In the previous section we showed that in order to prove Theorem \ref{Bias-Rank} for degree $d+1,$ the only new case we have to deal with is when $P:V^{[d+1]}\to\mF$ is multi-linear and $ \mcQ $ is a multi-linear tower on $V^{[d+1]}$ of degree $\le d.$ In this section we will prove the existence of an abundance of relative low rank approximations, and in the next sections we will "glue" them together to get a genuine relative low rank presentation. We start with a useful algebraic claim which we will need here and in later sections.
	
		\begin{claim}\label{mult-coord}
		Let $ \mcQ $ be a tower on $ V^{[k]} $ and $ \textbf{l}\in \mN^k$ a vector. Define $ \mcQ^{\otimes \textbf{l}} $ to be the tower on $ (V_1)^{l_1} \times \ldots \times (V_k)^{l_k}$ obtained by multiplying the first coordinate $ l_1 $ times, the second $ l_2 $ times and so on - i.e. a tower of the same height as $ \mcQ $ with
		\[ 
		\left( \mcQ^{\otimes \textbf{l}} \right)_i = \bigcup_{ \textbf{j}\in \prod_{t \in I_i} [l_t] } \mcQ_i( (x_{t,j_t})_{t \in I_i} ) .     
		 \]
		Define $ \mcQ^{\times \textbf{l}} $ to be the tower composed of the same equations, but on $ V^{[kl_1\ldots l_k]}, $ so it has $ l_1\ldots l_k $ as many layers as $ \mcQ $ (We order the layers corresponding to $ \mcQ_i $ arbitrarily within themselves, but keep the order otherwise).
		Then if $\mcQ$ is $(A(l_1\ldots l_k)^B,B,s)$-regular then both $ \mcQ^{\otimes \textbf{l}} $ and $ \mcQ^{\times \textbf{l}} $ are $ (A,B,s) $-regular.
	\end{claim}
	
	\begin{proof}
		We may assume without loss of generality that $ \mcQ $ is multi-linear. Because there are $ l_1\ldots l_k $ times as many maps, it is enough to show that the relative ranks are unchanged. Suppose we have a linear combination with 
		\[
			rk_{\mcQ^{\otimes \textbf{l}}_{<i}} \left( \sum_{\textbf{j} \in \prod_{t \in I_i} [l_t] } a_{\textbf{j}} \cdot \mcQ_i ( (x_{t,j_t})_{t \in I_i} ) \right) < rk_{\mcQ_{<i}} (\mcQ_i).
		\]
		For any  $\textbf{j}\in \prod_{t \in I_i} [l_t],$ plugging in $x_{\alpha,\beta} = 0$ for all $ \beta\neq j_\alpha, $ we get 
		\[
		rk_{\mcQ_{<i}} ( a_{\textbf{j}}\cdot \mcQ_i ) < rk_{\mcQ_{<i}} (\mcQ_i),
		\]
		which implies that $ a_{\textbf{j}} = 0.$ The proof for $  \mcQ^{\times \textbf{l}} $ is identical.
	\end{proof}

	Now we are ready to prove that we can find an abundance of low rank approximations of $ P $ on $ Z(\mcQ). $ The following is inspired by lemma 12 in \cite{Mi}.
	
	\begin{proposition}[Low rank approximations]\label{approx}
		There exist constants $A(d,H),B(d,H)$ such that the following holds: If  $ \mcQ $ is an $(A,B,s)$-regular multi-linear tower on $ V^{[d+1]}$ of degree $\le d$ and height $\le H,$ $ X=Z(\mcQ), $ and $P:V^{[d+1]}\to\mF$ is a multi-linear map satisfying
		$bias_X (P) \ge q^{-s},$ then there exists a set $E\subset (X_{d+1})^{100s}$ of density $\ge q^{-10s}$ such that for all $ t\in E $ there is a multi-linear function $\phi_t:V^{[d]}\to\mF^{100s}$  with
		\[
		\mP_{x\in X_t\cap \{\phi_t(x_{[d]})=0\}} (P(x) = 0) \ge 1-q^{-s},
		\]
		where $X_t = \{x\in X : x_{[d]}\in \bigcap_{i=1}^{100s} X(t_i)\}.$
	\end{proposition}

	\begin{proof}
	Choosing a non-degenerate bilinear form on $V^{d+1},$ we can write \linebreak
    $P(x) = A(x_{[d]})\cdot x_{d+1},$ where $A:V^{[d]}\to V^{d+1}$ is multi-linear. Apply Lemma \ref{Fubini} to get
		\begin{equation}\label{many-zeros}
			q^{-2s} \le \mE_{x_{[d]}\in X_{[d]}} \mE_{x_{d+1}\in X(x)} A(x_{[d]})\cdot x_{d+1} = \mP_{x_{[d]}\in X_{[d]}} \left[A(x_{[d]})\perp X(x_{[d]})\right],
		\end{equation}
  where $v\perp w$ means $v\cdot w =0.$
		Now, fix $x_{[d]}\in X_{[d]}.$ To test for the event above,  choose $ t\in (X(x_{[d]}))^l $ at random where $ l=100s $ and define
		\[
		\phi_t(x_{[d]}) = (A(x_{[d]})\cdot t_1,\ldots, A(x_{[d]})\cdot t_l).
		\]
		If $ A(x_{[d]})\perp X(x_{[d]}), $ then $\phi_t(x) = 0$ for any $t\in (X(x_{[d]}))^l$ and if $ A(x_{[d]})\not\perp X(x_{[d]}) $ then 
		\[
		\mP_{t\in X(x_{[d]})^l} (\phi_t(x) = 0) = q^{-l}.
		\]
		Therefore,
		\[
			\mE_{x_{[d]}\in X_{[d]}} 1_{A(x_{[d]})\not\perp X(x_{[d]})}\mE_{t\in (X(x_{[d]}))^l} 1_{\phi_t(x) = 0} \le q^{-l}.
		\]
		If $ A,B $ are sufficiently large, then by Claim \ref{mult-coord} and Lemma \ref{Fubini} we have 
		\[
		\mE_{t\in (X_{d+1})^l} \mE_{x\in X_t} 1_{A(x_{[d]})\not\perp X(x_{[d]})} 1_{\phi_t(x) = 0} \le q^{-l}+q^{-100s} \le q^{-50s}.
		\]
		By Markov's inequality, for $q^{-25s}$-a.e. $t\in X_{d+1}^l,$ we have
		\begin{equation}\label{false-zero}
			\mE_{x\in X_t} 1_{A(x_{[d]})\not\perp X(x_{[d]})} 1_{\phi_t(x) = 0} \le q^{-25s}.
		\end{equation}
		Also, by Lemma \ref{Fubini} together with inequality \eqref{many-zeros}   
		\[ 
		\mE_{t\in X_{d+1}^l}\mE_{x\in X_t} 1_{A(x_{[d]}) \perp X(x_{[d]})} \ge  q^{-3s}.
		\]
		Markov's inequality gives us
		\[
		\mP_{t\in X_{d+1}^l} \left[ \mE_{x\in X_t} 1_{A(x_{[d]}) \perp X(x_{[d]})} \ge q^{-5s}\right] \ge q^{-5s}. 
		\]
		So for a collection of $t\in X_{d+1}^l$ of density $\ge q^{-10s}$ we have both 
		\[
		\mE_{x\in X_t} 1_{A(x_{[d]}) \perp X(x_{[d]})} \ge q^{-5s}
		\]
		and inequality \eqref{false-zero}.
		For these $t,$ we have
		\begin{multline*}
			\mP_{x\in X_t\cap \{\phi_t(x_{[d]})=0\}} (P(x) = 0) \ge \mP_{x\in X_t} (A(x_{[d]}) \perp X(x_{[d]}) | \phi_t(x)=0) \\
			= \frac{\mP_{x\in X_t} (A(x_{[d]}) \perp X(x_{[d]}))}{\mP_{x\in X_t} (A(x_{[d]}) \perp X(x_{[d]}))+\mP_{x\in X_t} (A(x_{[d]}) \not\perp X(x_{[d]}) ,\phi_t(x)=0)} \\
			\ge \frac{q^{-5s}}{q^{-5s}+\mP_{x\in X_t} (A(x_{[d]})\not \perp X(x_{[d]}) ,\phi_t(x)=0)} 
			\ge \frac{q^{-5s}}{q^{-5s}+q^{-25s}} \ge 1-q^{-s}.
		\end{multline*}
	\end{proof}
	
	\section{Regularization}\label{regularization-sec}
	
	In order to improve our approximation from the last section to a genuine low rank representation, we will have to take the $ 100s $ maps $ \phi_t:V^{[d]}\to\mF^{100s} $ and replace them by a tower which is relatively regular together with the maps of $ X_t. $ In this section we describe an efficient algorithm for doing this. The following notion will come in handy:
	
	\begin{definition}
		Let $ \mcR=(\mcR_i)_{i\in[h]}, \mcQ=(\mcQ_j)_{j\in [h]}$ be towers on $ V^{[k]}$ and $ m_i = |\mcR_i|.$ We say that $ \mcR $ is $ (A,B,s) $-regular relative to $ \mcQ $ if for all $ i\in[h] $ we have 
		\[
		rk_{\mcQ\cup\mcR_{<i}} (\mcR_i) > A(m_i+m_{i+1}+\ldots+m_h+s)^B.
		\]
	\end{definition}

	Note that if $\mcQ,\mcR$ are two towers such that $\mcR$ is $(A,B,s)$-regular relative to $ \mcQ $ and $\mcQ$ is $(A,B,s+m)$-regular ($ m $ is the dimension of $ \mcR $), then $\mcQ\cup\mcR$ is $(A,B,s)$-regular with the layers ordered by having $ \mcQ $ on the bottom and $ \mcR $ on top.
	
	\begin{claim}[Regular decomposition]\label{Regularization}
		For any $A,B$ there exist constants $C(A,B,d,k),D(A,B,d,k)$ such that the following holds: If $\mcQ$ is a tower on $ V^{[k]} $ and $\mcR$ is a collection of full multi-affine maps of degree $\le d$  then there exist towers $\mcS_1,\ldots,\mcS_l$ such that:
		\begin{enumerate}
			\item $\mcS_i$ is $(A,B,s)$-regular relative to $ \mcQ $,
			\item $\mcS_i$ is of degree $\le d$ and height $\le 2^k,$
			\item $\dim(\mcS_i)\le C(s+\dim(\mcR))^D$ and
			\item we have a decomposition $Z(\mcQ \cup \mcR) = \bigsqcup_{i=1}^l  Z(\mcQ \cup \mcS_i). $ 
		\end{enumerate}
		In addition, if $ \mcQ,\mcR $ are multi-linear then we can take $ \mcS_1 $ to be multi-linear.
	\end{claim}

    Note that Theorem \ref{polyreg} is the above claim for homogeneous polynomials in the simplest case where $\mcQ$ is empty, and the proof we will now give of Claim \ref{Regularization} may be easily modified to prove Theorem \ref{polyreg}.
    
	\begin{proof}
		First, arrange $ \mcR $ in a tower $ \mcR = (\mcR_i:V^{I_i}\to\mF)_{i\in[2^k]} $, with the higher degree maps above the lower degree ones. As usual, denote $ m_i = |\mcR_i|. $  
		To the $i$-th layer we assign a number $n_i,$ according to a formula we will give later. We describe an algorithm for producing $ \mcS_i $ and then show that it has the desired properties. If for every $ i $
		\[  
		rk_{\mcQ\cup\mcR_{<i}} (\mcR_i) > A(n_i+s)^B,
		\]
		then we stop. Otherwise, there is some $0\neq a\in\mF^{m_i} $ such that 
		\[
		a\cdot \mcR_i = \sum_{j=1}^t q_j(x_{I_j}) r_j (x_{I_i\setminus I_j}) + \sum_{J \subsetneq I_i} u_J(x_J) + \sum_{f\in \mcQ\cup\mcR_{<i},\ I_f\subset I_i} f(x_{I_f})h_f(x_{I_i\setminus I_f}),
		\]
		where $t \le A(n_i+s)^B.$ Assume without loss of generality that $ a_{m_i} \neq 0. $ Let $ \mcR' $ be the tower obtained from $ \mcR $ by deleting $ R_{i,m_i}. $ In $Z(\mcQ\cup\mcR'),$ the value of $R_{i,m_i}$ is determined by the $q_j,r_j,u_J.$ We add these maps to the lower layers of $\mcR'$ and call the new tower $\mcR''.$  We can then split
		\[
		 Z(\mcQ\cup\mcR) = \bigsqcup_{i=1}^c \left( Z(\mcQ) \cap \{\mcR''=b_i\} \right) ,
		 \]
		where $ b_1,\ldots,b_c \in\mF^{\mcR''}$ are the values causing the maps in $ \mcR $ to vanish on\\
		$ Z(\mcQ) \cap \{\mcR''=b_i\}. $ Now we do the same thing for each of the towers $ \mcR''-b_i $ (with the same $n_i$ we chose at the start) and so on. Because at each stage we replace a map by maps of lower degree (and when handling a linear map we just delete it), this algorithm eventually terminates. We now give the definition of $n_i$ and check that the $ \mcS_i $ we end up with satisfy requirements 1-4. Set
		\[
		n_{2^k} = m_{2^k} ,\ n_i =  m_i + n_{i+1} \left[ 2A(s+n_{i+1})^B+2^d \right] .
		\]
		We prove inductively that the total number of polynomials passing through layer $i$ and above during this process  is $\le n_i.$ For $i = 2^k$ this is clear since during regularization we only add polynomials to lower layers. Now suppose this bound holds for $i.$ Any polynomials added to layer $i-1$ come from regularizing the higher layers. Each polynomial upstairs can contribute at most $2A(s+n_i)^B+2^d$ when regularizing, so the bound at layer $i-1$ is proved, assuming the bound at layer $i.$ Therefore, conditions 1 and 3 are satisfied. Conditions 2 and 4 clearly hold.  If $ \mcQ,\mcR $ are multi-linear then we can take $ b_1 = 0 $ at every stage, leaving us with $ \mcS_1 $ multi-linear. 
	\end{proof}
	
	\section{Fixing the approximation}\label{fixing}
	
	In this section, we will take the many approximations that we get from Proposition \ref{approx} and "glue" them together to get a single tower $ \mcR $ such that $ P $ vanishes identically on $ Z (\mcQ\cup\mcR). $ A central tool will be the following proposition (cf. proposition 5.1 in \cite{gt}):
	
	\begin{proposition}[Almost surely vanishing implies vanishing]\label{99-to-100}
		There exist constants $A(d,H),B(d,H),s(d)$ such that if $ \mcQ $ is an $(A,B,s)$-regular tower on $ V^{[d+1]} $ of degree $\le d$ and height $\le H$ and $f:V^{[d+1]}\to\mF$ is a polynomial map (not even necessarily multi-affine) of degree $ \le d+1 $ satisfying
		\[
		 \mP_{x\in Z(\mcQ)} \left(f(x) = 0\right) \ge 1-q^{-s},
		 \]
		then $f(x) = 0$ for all $x\in Z(\mcQ).$
	\end{proposition}
	
	Let $ \mcQ=(\mcQ_i)_{i\in[h]} $ where $ \deg(\mcQ_i) = d_i $ and write $ X = Z(\mcQ) $ for short. In order to prove this lemma we will need to count parallelepipeds based at any point $ x\in X. $ Given $x\in X$ and an integer $ l, $ for each layer $ i\in [h] $ we define a layer  $ (\mcQ_{x,l})_i $ of maps $ (V_1)^l\times\ldots\times (V_{d+1})^l\to\mF $ by
	\[
		(\mcQ_{x,l})_i := \{\mcQ_i(x+\omega\cdot t): \omega\in\{0,1\}^l,0<|\omega|\le d_i\},   
	\]
	and $ \mcQ_{x,l} $ is the resulting tower. Writing $ X_{x,l} = Z(\mcQ_{x,l}), $ note that 
	\[
	 X_{x,l} := \{t_1,\ldots,t_l\in V^{[d+1]} : x+\omega\cdot t\in X\ \forall \omega\in\{0,1\}^l\}.
	 \]
	We will also need to count parallelepipeds containing an additional point of the form $ x+t_1. $ Given $ t_1\in X-x $ and $ c\in[l], $ set 
	\[
	X_{x,l,c}(t_1) = \{t_2,\ldots,t_l\in V^{[d+1]} : (t_1-\ldots-t_c,t_2,\ldots,t_l)\in X_{x,l}\}.
	\]
	\begin{lemma} \label{pped-rank}
		If $\mcQ$ is $(Cl^{(d+1)D},D,s)$-regular then $\mcQ_{x,l}$ is an $(C,D,s)$-regular tower of the same degree and height. Also, if $ C,D $ are sufficiently large depending on $ A,B,d,H,l $ then for $ q^{-s} $ a.e. $ t_1\in X-x $ we have $ X_{x,l,c}(t_1) = Z(\mcR) $ where $ \mcR $ is an $ (A,B,s) $-regular tower of the same degree as $ \mcQ $, height $ \le Hl^{d+1} $ and 
		\[  
		 \dim(\mcR)= \dim(\mcQ_{x,l}) -\dim(\mcQ).
		\] 
	\end{lemma}
	
	\begin{proof}
		As usual, let $ \tilde{\mcQ} $ denote the multi-linear tower corresponding to $ \mcQ. $ The multi-linear tower corresponding to $ \mcQ_{x,l} $ is then $ \tilde{\mcQ}_{0,l}, $ so we can assume without loss of generality that $ \mcQ $ is multi-linear and $ x=0. $
		By Claim \ref{mult-coord}, the tower $ \mcQ^{\otimes (l,\ldots,l)} $ is $(A,B,s)$-regular. The maps in $ \mcQ_{x,l} $ are linear combinations of the maps in $  \mcQ^{\otimes (l,\ldots,l)}, $ so we just need to verify that they are linearly independent. Let $ Q\in\mcQ $ be a map of degree $ e $, assume w.l.o.g. that $ Q:V^{[e]}\to\mF, $ and suppose (to get a contradiction) that we have a non-trivial linear combination
		\begin{equation}\label{linear-independence}
			\sum_{\omega\in\{0,1\}^l,0<|\omega|\le e} a_\omega Q(\omega\cdot t) = 0.
		\end{equation}  
		 Let $ \omega_0 $ be the largest element with respect to the lexicographic ordering on $ \{0,1\}^l $ such that $ a_{\omega_0} \neq 0 $ and choose $ i_1,\ldots,i_e\in\omega_0 $ such that for every other $ \omega $ with $ a_\omega \neq 0 $ we have $ \{i_1,\ldots,i_e\} \not\subset\omega. $ Then the coefficient of $ Q(t_{i_1}(1),\ldots,t_{i_e}(e)) $ in the left hand side of equation \eqref{linear-independence} is $ a_{\omega_0} \neq 0, $ contradiction. Now we turn to the second part of the lemma. Since invertible affine transformations do not affect regularity, the variety 
		 \[ 
		 \{(t_1,\ldots,t_l): (t_1-\ldots-t_c,t_2,\ldots,t_l)\in X_{x,l}\}
		  \]
		  is the zero locus of a $ (C,D,s) $-regular tower on $ (V_1)^l\times\ldots\times (V_{d+1})^l. $ This tower remains regular if we split up the layers into at most $l^{d+1}$ layers so that it is a tower on 
		  \[ 
		  V_1\times\ldots\times V_1\times\ldots\times V_{d+1}\times\ldots\times V_{d+1}.
		   \]
		   Then by Lemma \ref{Derivatives} we get that for $ q^{-s} $-a.e. $ t_1\in X-x $ the variety 
		   \[
		   \{(t_2,\ldots,t_l): (t_1-\ldots-t_c,t_2,\ldots,t_l)\in X_{x,l}\}
		   \]
		   is the zero locus of a tower $ \mcR $ which has the required properties. 
	\end{proof}

	We now use this to prove the proposition from the start of this section.
	
	\begin{proof}[Proof of Proposition \ref{99-to-100}]
		Let $G = \{x\in X : f(x) = 0\},$ and $B = X\setminus G$ be the good and bad points of $ X, $ respectively. Fix $x\in X.$ Since $ f $ is a polynomial of degree $ \le d+1, $ for any $ t_1,\ldots,t_{d+2}\in V^{[d+1]} $ we have
		\[
		\sum_{\omega\in\{0,1\}^{d+2}} (-1)^{|\omega|}f(x+\omega\cdot t) = 0.
		\]
		This means that if we can find $ t\in X_{x,d+2}$ with $ x+\omega\cdot t\in G $ for all $ 0\neq \omega \in \{0,1\}^{d+2} $ then $ x\in G $ as well. Since the previous claim (together with Lemma \ref{Atom-size}) implies in particular that $ X_{x,d+2} \neq \emptyset, $ it is enough to show the following bound for any fixed $0\neq\omega\in\{0,1\}^{d+2}$
		\[
		\mE_{t\in X_{x,d+2} } 1_B (x+\omega\cdot t) \le 2^{-(d+2)}.
		\]
		By reordering the $ t_j $, assume without loss of generality that $\omega = (1^c,0^{d+2-c}).$
		By replacing $ t_1 $ with $ t_1-t_2-\ldots-t_c, $ the left hand side becomes
		\[
		\mE_{t,(t_1-t_2-\ldots-t_c,t_2,\ldots,t_{d+2}) \in X_{x,d+2}} 1_B (x+t_1).
		 \]
		By Lemma \ref{pped-rank}, the conclusion of Lemma \ref{Fubini} holds (the proof is the same), meaning
		\[ 
		\mE_{t,(t_1-t_2-\ldots-t_c,t_2,\ldots,t_{d+2}) \in X_{x,d+2}} 1_B (x+t_1) \le \mE_{t_1\in X-x}1_B (x+t_1)+q^{-s} \le q^{1-s}.
		 \]
		This proves the proposition for $ s \ge 3+d. $ 
 	\end{proof}
	
	We now prove a claim which will allow us to glue the various approximations we get in Lemma \ref{approx} together. 

	\begin{claim} \label{gluing}
		There exist $A(d,H),B(d,H),s(d)$ such that if $ f:V^{[d+1]}\to\mF $ is multi-linear, $\mcQ,\mcG_1, \mcG_2$ are multi-linear towers on $V^{[d+1]}$ satisfying $ f\restriction_{Z(\mcQ\cup\mcG_j)} = 0  $ and $ \mcQ\cup\mcG_1\cup\mcG_2 $ is $ (A,B,s) $-regular of degree $ \le d $ and height $ \le H, $ then for any $ i\in [d+1] $ we have 
		\[ 
		f\restriction_{ \{x\in Z(\mcQ)\ :\ x_{[d+1]
			\setminus \{i\}}  \in Z(\mcG_1\cup\mcG_2)_{[d+1]
			\setminus \{i\}} \} } \equiv 0.
		 \]
	\end{claim}

	\begin{proof}
		Denote $ Z = Z(\mcQ)\ ,\ Z_j = Z(\mcG_j). $ For any $ x\in (Z\cap Z_1 \cap Z_2)_{[d+1]\setminus \{i\}} $ we have $ f(x,\cdot) \restriction _{(Z\cap Z_j)(x)} = 0 $ for $ j=1,2 $ so by multi-linearity,
		\[ 
		f(x,\cdot) \restriction _{(Z\cap Z_1)(x)+(Z\cap Z_2)(x)} = 0.
		 \]
		 By Lemma \ref{Derivatives}, for $ q^{-2s} $-a.e. such $ x, $ the linear equations appearing in $ \mcQ(x)\cup \mcG_1(x)\cup \mcG_2(x) $ are linearly independent, in which case 
		 \[ 
		 (Z\cap Z_1)(x)+(Z\cap Z_2)(x) = Z(x).
		 \]
		 By Lemma \ref{Fubini},
		  \[ 
		\mP_{ \{x\in Z\ :\ x_{[d+1]
		 		\setminus \{i\}}  \in (Z_1\cap Z_2)_{[d+1]
		 		\setminus \{i\}} \} } (f(x) = 0) \ge 1-q^{-s}.
		  \]  
		 Applying Proposition \ref{99-to-100}, we get that this holds for any $ x $ in the above variety.
	\end{proof}

	\begin{corollary}\label{gluing-many}
		There exist $A(d,H),B(d,H),s(d)$ such that when $ f:V^{[d+1]}\to\mF $ is multi-linear, $\mcQ,\mcG_1,\ldots,\mcG_{2^l}$ are multi-linear towers on $V^{[d+1]}$ such that the $ \mcG_j $ only depend on the first $ l $ coordinates, $ \mcQ\cup\mcG_1\cup\ldots\cup\mcG_{2^l} $ is $ (A,B,s) $-regular of degree $ \le d $ and height $ \le H $, and 
		\[
		f\restriction_{Z(\mcQ)\cap Z(\mcG_j)}\equiv 0\ \forall j\in[2^l]
		\]
		then $ f\restriction_{Z(\mcQ)} \equiv 0.$
	\end{corollary}
	
	\begin{proof}
		
	By induction on $ l. $ For $ l = 0 $ this is clear. Suppose it holds up to some $ l $ and we are given $ \mcG_1,\ldots,\mcG_{2^{l+1}} $ depending on the first $ l+1 $ coordinates. Applying Claim \ref{gluing} to each pair $ \mcG_i,\mcG_{i+1} $ we get that 
	\[
	f\restriction_{ \{x\in Z(\mcQ)\ :\ x_{[d+1]
			\setminus \{l\}}  \in Z(\mcG_i\cup\mcG_{i+1})_{[d+1]
			\setminus \{l\}} \} } \equiv 0.
	  \]
	 Then apply the inductive hypothesis to 
	 \[
	  \mcQ, (\mcG_1\cup \mcG_2) _ {[l]},\ldots, (\mcG_{2^{l+1}-1}\cup \mcG_{2^{l+1}}) _ {[l]}.
	  \]		

	\end{proof}

	We summarize our results thus far with the following proposition, which brings us very close to proving Theorem \ref{Bias-Rank}.

	\begin{proposition}\label{vanishing}
		There exist constants $C(A,B,d,H),D(A,B,d,H)$ such that the following holds: If $ P:V^{[d+1]}\to\mF $  is multi-linear, $\mcQ$ is a multi-linear tower of degree $\le d$ and height $\le H$  such that
		\begin{enumerate}
			\item $bias_{Z(\mcQ)} (P) \ge q^{-s}$ and
			\item $\mcQ$ is $(C,D,s)$-regular,
		\end{enumerate}
		then there exists a multi-linear tower $ \mcR $ of degree $ \le d $ and height $ \le 2^{d+1} $ satisfying 
		\begin{enumerate}
			\item $ \dim(\mcR) \le C(1+s)^D, $
			\item $ \mcR\cup\mcQ $ is $ (A,B,s) $-regular and
			\item $ P\restriction_{Z(\mcR\cup\mcQ)} \equiv 0. $ 
		\end{enumerate}
	\end{proposition}

	\begin{proof}
		Denote $ X = Z(\mcQ) $ and let $E$ be the set from Proposition \ref{approx}. Given $ t\in E^{2^{d+1}}, $ let $ Y_i = \{x\in V^{[d+1]}: \phi_{t_i}(x_{[d]}) = 0 \} $ - so $P$ vanishes $q^{-s}$-a.e. on $X_{t_i}\cap Y_i.$ Applying Lemma \ref{Regularization}, there are towers $ \mcR_{i,1},\ldots, \mcR_{i,l}$ which are $ (A,B,s) $-regular relative to the tower defining $ X_{t_i}, $ satisfy $ \dim(\mcR_{i,j}) \le F(1+s)^G $ for some constants $ F(A,B,d),G(A,B,d) $ and such that
		\[ 
		X_{t_i}\cap Y_i = \bigsqcup_{j=1}^l (X_{t_i}\cap Y_{i,j}),
		 \]
		where $ Y_{i,j} = Z(\mcR_{i,j}). $ By averaging, there is some $ j\in [l] $ such that $P$ vanishes $q^{-s}$-a.e. on $X_{t_i}\cap Y_{i,j}$ - replace $ Y_i $ by this $ Y_{i,j}. $  Now suppose that $ X_{t_i} $ is $ (A,B,s+F(1+s)^G) $-regular (we will soon explain why we can choose $ t $ such that this happens) and note that $ X_{t_i}\cap Y_i $ is $ (A,B,s) $-regular. Applying Lemma \ref{affine-to-linear}, we can replace $ Y_i $ by $ \tilde{Y_i} $ and still have $ \mP_{x\in X_{t_i}\cap Y_i} \ge 1-q^{1-s}. $ By Lemma \ref{99-to-100},\\ $ P\restriction _{X_{t_i}\cap Y_i} \equiv 0. $ Setting $ Y= Y_1\cap\ldots\cap Y_{2^{d+1}}$ we have $ P\restriction_{X_{t_i}\cap Y} \equiv 0 $ for all $ i\in[2^{d+1}]. $ After applying Lemma \ref{Regularization} to $ Y $ relative to $ X_t := X_{t_1}\cap\ldots\cap X_{t_{2^{d+1}}} $  and assuming $ X_t $ is $ (A,B,s+F(2^{d+1} F(1+s)^G)^G) $-regular we get that $ X_t\cap Y $ is $ (A,B,s) $-regular. By Corollary \ref{gluing-many}, we conclude that $ P\restriction_{X\cap Y} \equiv 0 $ as desired.
		
		Now we need to justify our assumptions on the regularity of $ X_t. $ Writing $ n = 100s,$ Claims \ref{Derivatives} and \ref{mult-coord} imply that if $ \mcQ $ is $ (C,D,s) $-regular for sufficiently large $ C(A,B,d,H),D(A,B,d,H), $ then for $ q^{-2^{d+2} n} $-a.e. $ t\in \left( X_{d+1}^n \right)^{2^{d+1}} $ the variety $ X_t $ is $ (A,B,s+F(2^{d+1} F(1+s)^G)^G) $-regular as desired (in which case $ X_{t_i}, $ which is only defined by some of the equations, also must have the regularity we wanted). By Lemma \ref{approx}, $ E $ has density at least $ q^{-n} $ in $ X_{d+1} ^n, $ so $ E^{2^{d+1}} $ has density at least $ q^{-2^{d+1} n} $ in $  \left( X_{d+1}^n \right)^{2^{d+1}} $ which means that there must be some $ t\in E^{2^{d+1}} $ for which $ X_t $ has the desired regularity.
	\end{proof}

	\section{Nullstellensatz}\label{nullstellensatz-sec}
	
	In this section we finish the proof of Theorem \ref{Bias-Rank},  by proving a nullstellensatz for regular towers. This is similar to section 3 of \cite{Mi}, but by restricting our attention to regular towers we can prove a stronger result for them. We start with the following definition. 

	\begin{definition}
		We give $V^{[k]}$ a graph structure where two points $ x,y \in V^{[k]} $ are connected by an edge if there exists some $ i\in [k] $ such that $ x_{[k]\setminus \{i\}} = y_{[k]\setminus \{i\}}, $ i.e. they differ in at most one coordinate.
	\end{definition}

	We will need the following technical lemma.
	
	\begin{lemma}\label{connected}
		There exist constants $A(d,H),B(d,H)$ such that if $ \mcQ $ is an $(A,B,s)$-regular tower on $V^{[d+1]}$  of degree $\le d$ and height $\le H,$ $ X=Z(\mcQ) $ and $ I\subset [d+1] $ then for $ q^{-s} $-a.e. $ x\in X, $ for a.e. $ y\in X\cap X(x_{\ge I}) $ (with $ X(x_{\ge I}) := \cap_{I\subset J} \left( X(x_J)\times V^J \right) ), $ the points $ x_I,y_I $ are connected by a path in $ X( y_{[d+1] \setminus I } ). $
	\end{lemma}
	
	\begin{proof}
		A sufficient condition for there to be a path in $ X( y_{[d+1] \setminus I } ) $ between $x_I,y_I$ is that there exists $z $ such that for all $ K\subset I $ we have 
		\[
		(x_K , z_{I\setminus K}) , (y_K , z_{I\setminus K}) \in X(y_{[d+1] \setminus I }).
		\]
		Take the following variety
		\[ 
		Y = \left\lbrace 	(x_K , z_{I\setminus K} , y_{[d+1] \setminus I}) , (y_K , z_{I\setminus K} , y_{[d+1] \setminus I}), (x_J,y_{[d+1] \setminus J}) \in X:\ K\subset I\subset J \right\rbrace. 
		 \]
		The equations defining it are partial to those of $ \mcQ^{\times 3}, $ so by Claim \ref{mult-coord} we have $ Y=Z(\mcG) $ fo a  $ (C,D,s) $-regular tower of degree $ \le d $ and height $ \le 3^{d+1} H $ as long as $ \mcQ $ is $ (C(3^{d+1})^D,D,s) $-regular. By Lemmas \ref{Derivatives} and \ref{Fubini}, for $ q^{-s} $-a.e. $ x\in X, $ for $ q^{-s} $-a.e. $ y \in X\cap X(x_{\ge I}), $ the variety $ Y(x,y) $ (which is composed of the $ z'$s we want) is $ (C',D',s) $-regular. In particular, by Lemma \ref{Atom-size} it is non-empty.  
	\end{proof}
	
	Now we are ready to prove:
	
	\begin{theorem}[Nullstellensatz]\label{Nullstellensatz}
		There exist constants $A(d,H),B(d,H),s(d)$ such that if $\mcQ$ is an $(A,B,s)$-regular multi-linear tower on $V^{[d+1]}$ of deg $\le d$ and height $ \le H $ and $P: V^{[d+1]}\to\mF$  is a multi-linear map which vanishes $ q^{-s} $-a.e. on $ Z(\mcQ), $ then $ rk_\mcQ (P) = 0. $ 
	\end{theorem}

	The proof of this theorem (like everything else in this paper) will be by induction. For the inductive step we will use the following simple corollary many times.
	
	\begin{corollary}\label{single-rep}
			 If $ R_1,R_2:V^I\to\mF $ are multilinear maps, $ \mcQ_1\cup \mcQ_2 $ is a multi-linear tower on $ V^{[d+1]} $ for which the conclusion of Theorem \ref{Nullstellensatz} holds, and
			\[ 
			\mP_{x\in Z(\mcQ_1\cup\mcQ_2) }[R_1(x) = R_2(x)] \ge 1-q^{-s}
			 \]
			 then there exists a multi-linear map $ R:V^I\to\mF $ with $ R\restriction_{Z(\mcQ_1)} \equiv R_1\restriction_{Z(\mcQ_1)} $ and $ R\restriction_{Z(\mcQ_2)} \equiv R_2\restriction_{Z(\mcQ_2)}. $
	\end{corollary}
	\begin{proof}[Proof of corollary]
		By Theorem \ref{Nullstellensatz}, $ rk_{\mcQ_1\cup\mcQ_2} (R_1-R_2) = 0. $ Therefore there exist multi-linear maps $ f,g:V^I\to\mF $  with $ rk_{\mcQ_1} (f) = rk_{\mcQ_2} (g) = 0 $ and $ R_1-R_2 = f+g. $ Setting $ R = R_1-f = R_2+g$ does the job. 
	\end{proof}
	
	\begin{proof}[Proof of Theorem \ref{Nullstellensatz}]
		By Lemma \ref{99-to-100}, we have that $ P\restriction_{Z(\mcQ)} = 0. $ The proof will proceed by induction on the degree of $ \mcQ, $ and then on the number of layers of maximal degree. For the base case $ \mcQ = \emptyset $ this is trivial. For the inductive step, choose some $ Q_0\in\mcQ $ with a maximal size index set $ I, $ and write $ \mcQ' = \mcQ\setminus\{Q_0\}. $ Apply Lemma \ref{connected} to the variety $ X =  Z(\mcQ')\cap \{Q_0 = 1\}. $ Choose $ x^1,\ldots,x^{2^{d+1}} \in X $ such that for each $ j, $ $ x^j $ satisfies the conclusion of Lemma \ref{connected} with $ q^{-s} $-a.e. $ y\in X\cap X(x^j_{\ge I}) $ and such that $ \mcQ \cup \mcQ(\textbf{x}_{\ge I})$ is $ (C,D,s) $-regular, where $ \mcQ(\textbf{x}_{\ge I}) := \bigcup_{j=1}^{2^{d+1}} \mcQ(x^j_{\ge I}). $ The reason we can choose such $ x^j $'s is that the first condition holds $ 2^{d+1}q^{-s} $-a.s. by a union bound and the second one holds (as long as $ A,B $ are sufficiently large) $ q^{-s} $-a.s. by Claim \ref{mult-coord} and Lemma \ref{Derivatives}. Let $ \mcQ^j = \mcQ'\cup \mcQ(x^j_{\ge I}) $ and define $ R^j(y_{[d+1]\setminus I}) = P(x^j_I,y_{[d+1]\setminus I}). $		
		We claim that
		\[ 
		\mP_{y\in Z(\mcQ^j)} [P(y) = R^j(y_{[d+1]\setminus I})Q_0(y_I)] \ge 1-q^{-s}.
		 \]
 		Note that
		\[ 
		Z(\mcQ^j) = \bigsqcup_{a\in\mF} Z(\mcQ^j)\cap \{Q_0 = a\},
		 \]
		 and both sides vanish on $ Z(\mcQ) $ by assumption, so by scaling one of the coordinates it is enough to prove that
		 \[ 
		 \mP_{y\in Z(\mcQ^j) \cap \{Q_0 = 1\}} [P(y) = R^j(y_{[d+1]\setminus I})Q(y_I)] \ge 1-q^{-s}.
		  \]
		  Recall that by our definitions, $ Z(\mcQ^j) \cap \{Q_0 = 1\} =  X\cap X(x^j_{\ge I}). $ Now, for any $ y\in X\cap X(x^j_{\ge I}) $ we have equality at the point $ (x^j_I,y_{[d+1]\setminus I})\in Z(\mcQ^j). $ Also, for $ q^{-s} $-a.e. such $ y, $ the points $ x^j_I,y_I $ are connected in $ X(y_{[d+1]\setminus I}). $ Each time two points in $ X $ differ by a single coordinate in $ I, $ the values of both sides do not change, because the difference of the two points is in $ Z(\mcQ). $  This means that 
		 \[ 
		 \mP_{y\in X\cap X(x^j_{\ge I})} [P(y) = R^j(y_{[d+1]\setminus I})Q_0(y_I)] \ge 1-q^{-s}
		 \]
 		as desired, implying the same for $ Z(\mcQ^j). $
		
		By Lemma \ref{99-to-100}, we get that $  P(y) = R^j(y_{[d+1]\setminus I})Q(y_I) $ for all $ y\in Z(\mcQ^j). $ The following claim will complete the inductive step.
		
		\begin{claim}
			There exist constants $ C(d,H),D(d,H) $ such that if $ \mcF,\mcG^j $ for $ j\in[2^k] $ are towers on $ V^{[d+1]}, $ of degree $ \le d $ and height $ \le H, $ such that the $ \mcG^j $ depend only on the first $ k $ coordinates, and $ Q:V^I\to\mF,R^j:V^{[d+1]\setminus I}\to\mF $ are multi-linear maps, where
			\begin{enumerate}
				\item $ P(y) = R^j(y_{[d+1]\setminus I})Q(y_I)\ \forall y\in Z(\mcF \cup \mcG^j),$
				\item $ \mcF \cup \{Q\} \cup \bigcup_{j\in[2^k]} \mcG^j $ is $ (C,D,s) $-regular and
				\item the nullstellensatz holds for $ \mcF \cup \bigcup_{j\in[2^k]} \mcG^j $ and towers contained in it,
			\end{enumerate}
			    then there exists a multi-linear $ R:V^{[d+1]\setminus I}\to\mF $ such that $  P(y) = R(y_{[d+1]\setminus I})Q(y_I) $ for all $ y\in Z(\mcF). $
		\end{claim}
	
		\begin{proof}[Proof of claim]
			 We prove this by induction on $ k. $ For $ k = 0 $ there is nothing to prove. Now suppose we have proved it up to $ k $ and  $ \mcG^j $ are all defined on the first $ k+1 $ coordinates. By looking at $ Z(\mcF\cup\mcG^i\cup\mcG^{i+1}) $, we get $ R^i(y_{[d+1]\setminus I})Q(y_I) = R^{i+1}(y_{[d+1]\setminus I})Q(y_I) $ on $Z(\mcF\cup\mcG^i\cup\mcG^{i+1}). $ By Lemma \ref{Derivatives} (together with Lemma \ref{Atom-size}), for $ q^{-s} $-a.e. $ y_{[d+1]\setminus I}\in Z(\mcF\cup\mcG^i\cup\mcG^{i+1})_{[d+1]\setminus I}, $ there exists $ y_I $ such that $ y\in Z(\mcF\cup\mcG^i\cup\mcG^{i+1})\cap\{Q(y_I)=1\}, $ in which case $ R^i(y_{[d+1]\setminus I}) = R^{i+1}(y_{[d+1]\setminus I}). $ By Corollary \ref{single-rep}, there exists $ S^i $ such that $ P\restriction_{Z(\mcF \cup \mcG^j)} = S^i(x_{[d+1]\setminus I})Q(x_I) $ for $ j=i,i+1. $ Applying Lemma \ref{gluing}, we get that $  P\restriction_{Z(\mcF \cup (\mcG^i \cup \mcG^{i+1})_{[k]})} = S^i(x_{[d+1]\setminus I})Q(x_I). $ Now apply the induction hypothesis to the collection $  (\mcG^i \cup \mcG^{i+1})_{[k]} $ and we are done.
		\end{proof}
		
		Applying the above claim with $ \mcF = \mcQ' $ and $ \mcG^j = \mcQ(x^j_{\ge I}) $ (Note that the equations coming from $ x^j_{\ge I} $ are of deg $ < |I|, $ so the induction hypothesis applies to $ \mcQ'\cup\bigcup_{j\in[2^{d+1}]} Q^j $ and towers partial to it) we get that $ P(y) = R(y_{[d+1]\setminus I})Q(y_I) $ on $ Z(\mcQ'). $ By the inductive hypothesis, $ rk_{\mcQ'} \left( P(y) - R(y_{[d+1]\setminus I})Q(y_I) \right) = 0 $ and this completes the inductive step. 	
		\end{proof}
	
		\section{Reduction to the multi-linear case}\label{multilinear-reduction}
		
		In this section, we show that Theorems \ref{Bias-rank-poly} and \ref{null-poly} follow from their multi-linear analogues - Theorem \ref{Bias-Rank} and Theorem \ref{Nullstellensatz}. In order to do this, we will show that both relative bias and relative rank behave well when passing from polynomials to corresponding multi-linear maps. We start by recalling the polarization identity which associates a multi-linear map to a polynomial. The following is well known.
		
		\begin{claim}
			Given a polynomial $ P $ on $ V $ of degree $ d < char(\mF), $ the map \linebreak
   $ \bar P: V^d\to\mF $ defined by $ \bar P(x_1,\ldots,x_d) = \frac{1}{d!}\nabla_{x_1}\ldots\nabla_{x_d} P(0)$ is multilinear and satisfies $ \bar P(x,\ldots,x) = \tilde P(x) $ (the homogeneous part of $ P $ of highest degree).
		\end{claim} 
		\begin{definition}
			Given a polynomial tower $ \mcQ $ on $ V, $ of degree $ \le d < char(\mF) $ and height $ \le H, $ there exists a symmetric multi-linear tower $ \mcQ(k) $  of degree $ \le d $ and height $ \le 2^k H $ on $ V^k $ corresponding to it. For a layer $ \mcQ_i $ of degree $ e $ and a subset $ E\subset [k] $ of size $ e, $ we get $ \mcQ(k)_{i,E}:V^E\to\mF $ given by $ \bar{\mcQ_i}. $ Conversely, given a multi-linear tower $ \mcQ $ on $ V^k $ it induces a polynomial tower $ \mcQ\circ D $ on $ V $ by pre-composition with the diagonal map $ \mcQ\circ D = \mcQ(x,\ldots,x). $ 
		\end{definition}
		
		\begin{claim}\label{poly-to-multilinear}
			If $ \mcQ $ is an $ (A2^{kB},B,s) $-regular polynomial tower, then $ \mcQ(k) $ is $ (A,B,s) $-regular. Conversely if $ \mcQ(k) $ is $ (A2^{kB},B,s) $-regular and $ k \ge \deg(\mcQ) $, then $ \mcQ $ is $ (A,B,s) $-regular.
		\end{claim}
	
		\begin{proof}
			Assume without loss of generality that $ \mcQ $ is homogeneous. The polynomials of degree $ d_i $ in $ \mcQ $ give $ \binom{k}{d_i} \le 2^k $ polynomials in $ \mcQ(k), $ so we just need to check that 
			\[ 
			prk_{ \bar{\mcQ}_{<i} } (\bar{\mcQ}_i) \ge rk_{\mcQ_{<i}} (\mcQ_i).
			 \]
			This follows by plugging in the diagonal for any relative low partition rank combination. For the converse, we need to check that 
			\[ 
			prk_{ \bar{\mcQ}_{<i} } (\bar{\mcQ}_i) \le 2^k rk_{\mcQ_{<i}} (\mcQ_i).
			 \]
			 This follows from the polarization identity for any low relative polynomial rank combination. 
		\end{proof}
		
		We recall Theorem \ref{Bias-rank-poly}:
		
		\begin{theorem}[Relative bias implies relative low rank]
			There exist  constants $A(d,H), B(d,H)$ such that if $\mcQ$ is an $(A,B,s)$-regular tower of degree $\le d < char(\mF)$ and height $\le H$ and $P$ is a polynomial of degree $\le d$ with
			\[
			bias_{Z(\mcQ)}(P) \ge q^{-s}.
			\]
			then
			\[
			rk_\mcQ (P) \le A(1+s)^B. 
			\] 
		\end{theorem}
		
		We call $ \mcQ $ as above a tower of type $ (d,H) $ for short. The proof will be by induction on the degree $ d. $ For $ d = 1 $ it follows from classical Fourier analysis. Now we assume the theorem holds for degree $ d $ and explore some consequences.
		 
		\begin{lemma}[Regular varieties are of the expected size (d,H)]\label{atom-size-poly}
			There exist constants $A(d,H),B(d,H)$ such that for any $s>0$ and for any $(A,B,s)$-regular polynomial tower $\mcQ$ of type $ (d,H) $ defined on $ V $ of degree $ \le d, $ height $ \le H $ and dimension $m,$ we have 
			\[
			\left|q^m \frac{|Z(\mcQ)|}{|V|}-1\right|\le q^{-s}.
			\]
		\end{lemma}
	
		\begin{proof}
			Identical to that of Lemma \ref{Atom-size}.
		\end{proof}
	
		Now we will prove a lemma analogous to Lemma \ref{Derivatives}. 
		For a polynomial tower $ \mcQ $ and $ t\in  V $  we denote
		\[ 
		\del_t \mcQ = \{\del_t Q:Q\in\mcQ\}\ ,\ \mcQ_t = \mcQ\cup\del_t\mcQ.
		\]
		
	\begin{lemma}[Derivatives]\label{derivatives-poly}
		For any $C,D$ there exist constants $A(C,D,d,H),$ $B(C,D,d,H)$ such that if $\mcQ$ is an $(Ak^B,B,s)$-regular polynomial tower of type $ (d,H) $ then for $q^{-s}$-a.e. $t_1,\ldots,t_k\in Z(\mcQ(1))$ the tower $\mcQ_{t_1}\cup\ldots\cup\mcQ_{t_k}$ is $(C,D,s)$-regular. 
	\end{lemma}

	\begin{proof}
		By Claim \ref{poly-to-multilinear}, it is enough that $ (\mcQ_{t_1}\cup\ldots\cup\mcQ_{t_k} )(d) $ is $ (C2^{dD},D,s) $-regular. By the same claim, $ \mcQ(d+1) $ is $ (Ek^F,F,s) $-regular as long as $ A,B $ are large. By Claim \ref{mult-coord} applied with multiplicity vector $ l = (1,\ldots,1,k), $ the tower $ \mcQ' = (\mcQ(d+1))^{\otimes l} $ is $ (E,F,s) $-regular. Then by Lemma \ref{Derivatives} for $ E,F $ sufficiently large, it follows that for $ q^{-s} $-a.e. $t_1,\ldots,t_k\in Z(\mcQ(1)),$ the tower $ \mcQ'(t_1,\ldots,t_k) $ is $ (C,D,s) $-regular. $ \mcQ'(t_1,\ldots,t_k) = (\mcQ_{t_1}\cup\ldots\cup\mcQ_{t_k} )(d), $ so we are done.  
	\end{proof}
	
	\begin{corollary}[Fubini]\label{fubini-poly}
			There are constants $ A(d,H),B(d,H) $ such that if $ \mcQ $ is an $ (A,B,s) $-regular polynomial tower of type $ (d,H) $ and $ f $ is a one-bounded function, then we have 
			\[ 
			\left| \mE_{x,y\in Z(\mcQ)} f(x,y) - \mE_{t\in Z(\mcQ(1))} \mE_{x\in Z(\mcQ_t)} f(x,x+t) \right| \le q^{-s}.
			 \]
	\end{corollary}
	
	\begin{proof}
			Identical to the proof of Lemma \ref{Fubini}.
	\end{proof}
	
		For a function $ f:V\to\mC, $ and $ t\in V, $ denote the (multiplicative) derivative by $ D_t f(x) = f(x+t)\bar f(x). $ For $ t = (t_1\ldots t_k)\in V^k, $ write $ \mcQ_t = (\ldots(\mcQ_{t_1})_{t_2})\ldots)_{t_k}. $
	
	\begin{lemma}[Cauchy-Shcwarz]\label{CS}
			There exist $ A(d,H,k),B(d,H,k) $ such that if $ \mcQ $ is an $ (A,B,s) $-regular polynomial tower of degree $ \le d $ and height $ \le H, $ and $ f:V\to\mC $ is a one-bounded function, then
			\[
			| \mE_{x\in Z(\mcQ)} f(x) |^{2^k} \le Re \left( \mE_{t\in Z(\mcQ(k))}\mE_{x\in Z(\mcQ_t)} D_{t_1}\ldots D_{t_k} f(x) \right)  +q^{-s}.
			 \]
		\end{lemma}
		
	\begin{proof}
			By induction on $ k. $ For $ k = 1 $ this follows from Corollary \ref{fubini-poly}. Suppose it holds up to $ k-1. $ Then if $ A,B $ are sufficiently large we can apply Corollary \ref{fubini-poly} to get
			\begin{multline*}
				| \mE_{x\in Z(\mcQ)} f(x) |^{2^k} = |\mE_{x,y\in Z(\mcQ)} f(x) \bar f(y)|^{2^{k-1}} \le | \mE_{t\in Z(\mcQ(1))} \mE_{x\in Z(\mcQ_t)} D_t f(x) |^{2^{k-1}} + q^{-4s} \\
				\le  \mE_{t\in Z(\mcQ(1))} |\mE_{x\in Z(\mcQ_t)} D_t f(x)|^{2^{k-1}} + q^{-4s}.   
			\end{multline*} 
		If $ A,B $ are sufficiently large, then by Lemma \ref{derivatives-poly} the inductive hypothesis applies to $ q^{-4s} $-a.e. $ \mcQ_t. $ Therefore, 
		\begin{multline*}
			\mE_{t\in Z(\mcQ(1))} |\mE_{x\in Z(\mcQ_t)} D_t f(x)|^{2^{k-1}} \\
			\le Re \left( \mE_{t\in Z(\mcQ(1))} \mE_{s\in Z(\mcQ_t(k-1))} \mE_{x\in Z((\mcQ_t)_s)} D_sD_t f(x) \right)  +q^{-2s}.
		\end{multline*}

		Finally, using the fact that $ \mcQ_t(k-1) = (\mcQ(k))(t) $ and applying Lemma \ref{Fubini} for $ \mcQ(k), $ we get that
		\begin{multline*}
				Re \left( \mE_{t\in Z(\mcQ(1))} \mE_{s\in Z(\mcQ_t(k-1))} \mE_{x\in Z(\mcQ_{(t,s)})} D_sD_t f(x) \right) \\
				\le  Re \left( \mE_{t\in Z(\mcQ(k))}\mE_{x\in Z(\mcQ_t)} D_{t_1}\ldots D_{t_k} f(x) \right) + q^{-2s}, 
		\end{multline*} 
		 which completes the inductive step.
	\end{proof}
	
	\begin{corollary}\label{lower-tower}
		Theorem \ref{Bias-rank-poly} holds for polynomials of degree $ \le d+1 $ on towers of degree $ \le d. $
	\end{corollary}

	\begin{proof}
		When $ \deg(P) \le d $ we are already assuming this holds, so we deal with the case $ \deg(P) =d+1. $
		If $ bias_{Z(\mcQ)} (P) \ge q^{-s} $ then by Lemma \ref{CS} we get that $ bias_{Z(\mcQ(d+1))} (\bar P) \ge q^{-2s}. $ By Claim \ref{poly-to-multilinear}, the tower $ \mcQ(d+1) $ is regular, so we can apply Theorem \ref{Bias-Rank} to deduce that $ prk_{\mcQ(d+1)} (\bar P) \le A(1+s)^B. $ Plugging in the diagonal, we get that $ rk_\mcQ (P) \le A(1+s)^B. $
	\end{proof}
	
		To deal with towers of degree $ d+1, $ we introduce an auxiliary parameter like we did in the proof of Theorem \ref{Bias-Rank}:
		
		\begin{theorem}[Relative bias implies relative low rank]\label{bias-rank-poly-aux}
			There exist  constants $A(d,H,l), B(d,H,l)$ such that if $\mcQ$ is an $(A,B,s)$-regular tower of degree $\le d+1 < char(\mF)$ and height $\le H$ such that $ \mcQ_{\le H-l} $ is of degree $ \le d, $ and $P$ is a polynomial of degree $\le d+1$ with
			\[
			bias_{Z(\mcQ)}(P) \ge q^{-s}.
			\]
			then
			\[
			rk_\mcQ (P) \le A(1+s)^B. 
			\] 
		\end{theorem}
	
		A tower $ \mcQ $ as above is called a tower of type $ (d+1,H,l) $ for short. We have just seen this theorem holds for towers of type $ (d+1,H,0). $ Theorem \ref{Bias-rank-poly} is for towers of type $ (d+1,H,H). $
		
		\begin{proof}[Proof of Theorem \ref{bias-rank-poly-aux}]
			By induction on $ l. $ We have just seen the base case $ l=0. $ Suppose the theorem holds for towers of type $ (d+1,H,l) $ and $ \mcQ $ is now of type $ (d+1,H,l+1). $ The inductive assumption implies Lemma \ref{atom-size-poly} for $ \mcQ $ and $ \mcQ_{<H} $ (as in the proof of Lemma \ref{Atom-size}), so we get 
			\[ 
			\sum_{a\in\mF^{m_H}} bias_{Z(\mcQ_{<H})} (P+a\cdot \mcQ_H) \ge q^{-2s}.
			 \] 
			By sub-additivity of relative rank, there is at most a single $ a_0\in\mF^{m_H} $ with $ rk_{\mcQ_{<H}} (P+a\cdot \mcQ_H) \le A/2(m_H+s)^B, $ and if $ \deg(P) < \deg(\mcQ_H) $ then $ a_0 = 0. $ By the theorem applied to $ \mcQ_{<H} $ we get that for all $ a\neq a_0, $
			\[ 
			bias_{Z(\mcQ_{<H})} (P+a\cdot \mcQ_H) \le q^{-4s-m_H},
			 \]
			so $ bias_{Z(\mcQ_{<H})} (P+a_0\cdot \mcQ_H) \ge q^{-4s}. $ This gives
			\[ 
			rk_\mcQ (P) \le rk_{Z(\mcQ_{<H})}(P+a_0\cdot \mcQ_H) \le A(1+s)^B.
			 \]
		\end{proof}

	We are also ready to prove Theorem \ref{null-poly}:
	
	\begin{theorem}[Robust Nullstellensatz for regular collections]
		There exist constants $A(d,H),B(d,H),s(d)$ such that if $\mcQ = (Q_{i,j})_{i\in[h],j\in[m_i]} $ is an $(A,B,s)$-regular polynomial tower of deg $\le d < char(\mF)$ and height $ \le H, $ and $P$  is a polynomial of degree $ \le d $ which vanishes $ q^{-s} $-a.e. on $ Z(\mcQ), $ then there exist polynomials $ R_{i,j} $ satisfying $ \deg(R_{i,j}) + \deg(Q_{i,j}) \le \deg(P) $ and 
		\[ 
		P = \sum_{i,j} R_{i,j}Q_{i,j}.
		\] 
	\end{theorem}
	
	\begin{proof}
		By induction on $ \deg(P). $ For constant $ P $ there is nothing to prove, since we know by  Lemma \ref{atom-size-poly} that $ Z(\mcQ) \neq \emptyset. $ Suppose the theorem holds up to degree $ k-1 $ and $ \deg(P) = k. $ Suppose $ P $ vanishes $ q^{-s} $-a.e. on $ Z(\mcQ). $ By Lemmas \ref{CS} and \ref{Fubini}, $ \bar P $ vanishes $ q^{-s} $-a.e. on $ Z(\mcQ(k)). $ By Theorem \ref{Nullstellensatz}, $ rk_{\mcQ(k)} (\bar P) = 0. $ Plugging in the diagonal, this implies that $ rk_{\tilde\mcQ} (\tilde P) = 0 $ so there exist $ R_{i,j} $ with $ \tilde P = \sum R_{i,j} \tilde Q_{i,j}. $ The polynomial $ P' = P-\sum R_{i,j} Q_{i,j} $ has $ \deg(P') \le k-1 $ and also vanishes $ q^{-s} $-a.e. on $ Z(\mcQ). $ Applying the inductive hypothesis finishes the proof.
	\end{proof}

	\appendix

\section{Adaptations of results involving rank to relative rank}	
We formulate and prove the relevant parts of the Theorems in \cite{kz-extension, kz-uniform}, replacing rank with relative rank. Let $\mF$ be a field. For an algebraic $\mF$-variety $\mX$ 
we write $X(\mF):=\mX (\mF)$. To simplify notations we often write $X$ instead of $X(\mF)$. In particular we write $V:=\mV (\mF)$ when $\mV$ is a vector space and write $\mF^N$ for $\mA^N(\mF)$. 

For a $\mF$-vector space $\mV$ 
we denote by $\mcP_{\bar d}(\mV)$ the algebraic variety of tuples polynomials $\bar P= (P_1, \ldots, P_c)$ on $\mV$ of degrees $\le \bar d$ and  by $\mcP_{\bar d}(V) $ the set of tuples of polynomials functions $\bar P:V\to \mF^c$ of degree $\leq \bar d$.
  We always assume that   $d < char(\mF)$, so the restriction  map $\mcP_{\bar d}(\mV)(\mF)\to \mcP_{\bar d}(V)$ is a bijection. 
  
  We recall some of the definitions in \cite{kz-extension}.
\begin{definition}\label{Af}
\leavevmode
\begin{enumerate}
\item For $m\geq 1$ and a $\mF$-vector space $V$, we denote by $\text{Aff} _m(\mV)$ the algebraic variety of affine maps $\phi :\mA ^m\to \mV$ and write $\text{Aff} _m(V) := \text{Aff} _m(\mV) (\mF)$.
 \item We define an algebraic morphism $\ti \kk _{\bar P}: \text{Aff} _m(\mV) \to \mcP_{\bar d}(\mA ^m)$ by $\ti \kk _{\bar P} (\phi):= \bar P \circ \phi$, and denote  by $\kk _{\bar P}$ the corresponding map  $\text{Aff} _m(V) \to \mcP_{\bar d}(\mF^m) $.
\end{enumerate}
\end{definition}

\begin{definition}\label{un}
 A map $\kk :M\to N$ between finite sets is {\em $\ep$-uniform}, where $\ep >0$, if for all $n\in N$ we have
$$\left||N||\kk ^{-1}(n) |-|M| \right|\leq \ep |M|.$$
\end{definition}

\begin{theorem}\label{A1}
 There exist constants $ A(d,h),B(d,h) $ such that for any $m, s\geq 1,$ any finite field $ \mF=\mF_q, $ and any $ (Am^B,B,s) $-regular tower $ \bar{P} $ of degree $ \le d <char(\mF) $ and height $ \le h, $ the map $\kk _{\bar P}:\text{Aff} _m(\mV) \to \mcP_{\bar d}(\mA ^m)$ is $q^{-s}$ uniform.
\end{theorem}

As a corollary we obtain (as in \cite{kz-extension})

\begin{theorem}
For any  $m \geq 1$,  any algebraically closed  field $\mF$ of characteristic zero or $>d$,  any $\mF$-vector space $V$ and any $ (Am^B,B,2) $-regular tower $\bar P$ of degree $ \le d $ and height $ \le h, $ where $ A,B $ are the constants of Theorem \ref{A1}, we have
\begin{enumerate}
\item The map $\kk _{\bar P}$ is surjective,
\item All fibers of the morphism  $\ti \kk _{\bar P}$ are of the same dimension,
\item The morphism  $\ti \kk _{\bar P}$  is flat and
\item $\bar P\subset \mF[V^\vee]$ is a regular sequence.
\end{enumerate} 
\end{theorem}


Let $\epsilon>0$ we say that a property holds for $\epsilon$-a.e. $s \in S$ if it holds for all but $(1-\epsilon)|S|$ of the elements in $S$. 

As in \cite{kz-extension},  Theorem \ref{A1} follows from the following result:

\begin{theorem}\label{need}
Let $d \ge 0$.   There exist $A(d,h),B(d,h)$ such that for any $s,m>0$ the following holds: Let $\mF=\mF_q$ be a finite field of characterisitc $>d$, let $\bar P$ be an $ (Am^B,B,s) $-regular tower of degree $ \le d $ and height $ \le h $ composed of polynomials $ (P_1,\ldots,P_c) $ of degrees $ (d_1,\ldots,d_c) $ on $V=\mF^n.$ Then: 
\begin{enumerate}
\item  For any collection of polynomials $\bar R=(R_i)_{ 1 \le i \le c}$, with $R_i:\mF^m \to \mF$  of degree $d_i$,  there exist an affine map $w:\mF^m \to \mF^n$ such that 
$\bar P(w(x))=\bar R(x)$.  Furthermore, if we denote by $n_{\bar R}$ the number of such affine maps, then for any $\bar R_1, \bar R_2$ as above $|1-n_{\bar R_1}/n_{\bar R_2}| <q^{-s}$. 
\item  If $\bar P$ is homogeneous, then for any homogeneous collection $\bar R=(R_i)_{ 1 \le i \le c}$,  $R_i:\mF^m \to \mF$  of degree $d_i$,  there exist a linear map $w:\mF^m \to \mF^n$ such that $\bar P(w(x))=\bar R(x)$.  Furthermore, if we denote by $n_{\bar R}$ the number of such linear maps, then for any $\bar R_1, \bar R_2$ as above $|1-n_{\bar R_1}/n_{\bar R_2}| <q^{-s}$.
\end{enumerate}
\end{theorem}

\begin{proof}
We prove by induction on $d$. For $d=1$ the statement is clear. Assume $d>1$
We are given a filtered collection of polynomials $\mcP=\{P^{e,f}\}_{1\le e \le h, 0 \le f \le J_e}$.  We first make the following observation. Let $P$ be of degree $d$ . 
$P(t)=\sum_{I \in \mcI_k}a_{I}t_I$
where $\mcI_{d}(n)$ is the set of
ordered tuples $I=(i_1, \ldots, i_{d})$ with  $1 \le i_1 \le \ldots \le i_{d} \le n$, and $t_I= t_{i_1} \ldots t_{i_{d}}$.
Note that  for any polynomials $l(t)$ of degrees $<d$  we have that $P(t)+l(t)$ is of the same rank as $P$. 
We can write
\[
P(w(x)) = \sum_{I \in \mcI_{d}(n)} a_{I}w^I(x)
= \sum_{I \in \mcI_{d}(n)} a_{I}\sum_{l_1, \ldots, l_{d}=1}^m w^{i_1}_{l_1} \ldots w^{i_{d}}_{l_{d}}x_{l_1} \ldots x_{l_{d}},
\]
where $w^I = \prod_{i \in I} w^{i}$. 
For $( l_1, \ldots ,l_{d}) \in \mcI_{d}(m)$  the term  $x_{l_1} \ldots x_{l_{d}}$ has as coefficient
\[
Q_{( l_1, \ldots ,l_{d})}(w)= \sum_{\sigma \in S_{d}} \sum_{I \in \mcI(n)} a_{I}w^{i_1}_{l_{\sigma(1)}} \ldots w^{i_{d}}_{l_{\sigma(d)}}.
\]

Observe that restricted to the subspace
$w_{l_1} = \ldots = w_{l_{d}}$ we can write the above as
 \[
Q_{( l_1, \ldots ,l_{d})}(w) = d! P(w_{l_1}) + R(w)
\]
where $w_j=(w_j^1, \ldots, w_j^n)$, and $R(w)$ is of lower degree in  $w_{l_1}$.  

Now consider the filtered system $\mcP$. It gives rise to a filtered system 
\[
\{Q^{e,f}_{( l_1, \ldots ,l_{e})}(w)\}_{1\le e \le h, 0 \le f \le J_e, ( l_1, \ldots ,l_{d_e}) \in \mcI_{d_e}(m)}
\]
Given an  element $Q^{e,f}_{( l_1, \ldots ,l_{d})}$ at level $e$ of degree $d$, we restrict to a subspace as above and get 
$ Q^{e,f}(w)=d! P(w_{l_1}) + R^{e,f}(w)$ and $R(w)$ is of lower degree in  $w_{l_1}$.  But the same holds for all lower degree polynomials as well - restricted to this subspace they are of the form $j!P^{e,f}(w_{l_1}) + R^{e,f}(w)$ where $R^{e,f}$ of lower degree in  $w_{l_1}$.  Similarly this hold for a linear combination of elements at lever $h$.  This is sufficient since relative rank can only decrease when restricting to a subspace.  So we just need to make sure that the original relative rank beats the new number of polynomials at each level which is $ \le m^{O(d)}$. 

\end{proof}
 
We formulate a number of additional results, that were proved in the context of high rank collections of polynomials, and whose proofs can be easily adapted to the new notion of relative rank. 

We recall the definition of weakly polynomial function from \cite{kz-extension}
\begin{definition}\label{weak-def}\leavevmode
\begin{enumerate}
\item  Let  $V$ be a $\mF$-vector space  and  $X\subset V$. We say that a function $f:X \to \mF$ is {\it weakly polynomial} of degree $\leq a$ if the restriction $f_{|L}$ to any affine subspace  $L \subset X$ is a polynomial of degree $\leq a$.  
\item $X$ satisfies $\star ^a$ if any weakly polynomial function of degree $\leq a$ on $X$ is a restriction  of a polynomial function of degree $\leq   a$ on $V$.
\end{enumerate}
\end{definition}

\begin{theorem}\label{w}For any $d, a\geq 1$ there exist constants  $A(d,h,a),B(d,h,a)$ such that the following holds: For any field $\mF$ which is either finite with $|\mF|>ad$ or algebraically closed of characteristic zero or $>d$, a $\mF$-vector space $V$, an $ (A,B,2) $-regular tower  $\mcP$ of degree $\le d$ and height $ \le h,$  the subset $Z(\mcP)\subset V$ has the property $\star_a$.
\end{theorem}

Combined with the polynomial cost regularization of Theorem \ref{polyreg}  we obtain:

\begin{theorem}\label{w1}For any $d, a\geq 1$ there exists  $A=A(d,a)$ such the following holds:  for any field $\mF$ which is either finite with $|\mF|>ad$ or algebraically closed of characteristic zero or $>d$, a $\mF$-vector space $V$, a collection of polynomials  $\bar P=(P_i)_{ 1 \le i \le c}$ with $\deg P_i \le d$, there exists a collection of polynomials $\mcQ$ such that $\bar P \subset \mcI(\mcQ)$  such that  $Z(\mcQ)\subset V$ has the property $\star _a$, and $|\mcQ| \le Ac^A$. 
\end{theorem}

Finally we can also adapt the results from \cite{kz-uniform} to the setting of relative rank: Let $V$ be a vector space over a field $\mF$.  An $m$-cube  in a vector space $V$ is a collection $(u|\bar v),u\in V,\bar v\in V^m$ of $2^m$ points 
$\{ u+\sum _{i=1}^m \bo _iv_i\}$, $\bo _i\in \{0 ,1\}$.

For any map  $f:V\to H$ where $H$ is an abelian group we denote by $f_m$ the map from the set $C_m(V)$ of $m$-cubes to $H$ given by 
$$f_m(u|\bar v)=\sum _{\bar \bo \in \{ 0 ,1\}^m}(-1)^{|w|}f(u+\sum _{i=1}^m\bo _iv_i)$$
where $|\bo|=\sum _{i=1}^m\bo _i$.  For a subset  $X \subset V$ we denote 
 $C_m(X)$ the set of $m$-cubes in $V$ with all vertices in $X$. Note that in the case that $H=\mF$, where $\mF$ a prime field, functions $f:V \to \mF$ such that 
 $f_m$ vanishes on $C_m(V)$ are precisely polynomials of degree $<m$.

\begin{theorem}\label{testing-X-high}
Let $m, d>0$. There exists $C=C(d,m,h)$  such  that the following holds: For any $0<\epsilon$,  a finite field $\mF$ and an $\mF$-vector space $V$, a $ (C,C,2) $-regular tower $ \mcP $ of degree $ \le d $ and height $ \le h $ and any
 $f:X=Z(\mcP) \to H$  with $f_m(c)=0$ for $\epsilon$-a.e.  $c \in C_m(X)$,  there exists a  function $h:X \to H$  such that $h_m\equiv 0$, and  $h(x)=f(x)$ on $C \epsilon$-a.e. $x \in X$. \end{theorem}


\begin{thebibliography}{99}	
\bibitem{akz}Adiprasito, K. , Kazhdan, D., Ziegler, T {\em  On the Schmidt and analytic ranks for trilinear forms}.  arXiv:2102.03659 (2021).
\bibitem{ah}  Ananyan, T., Hochster, M.  {\em Small subalgebras of polynomial rings and Stillman's Conjecture}. J. Amer. Math. Soc. 33 (2020), 291-309. 
\bibitem{bl} Bhowmick A., Lovett S. {\em Bias vs structure of polynomials in large fields, and applications in effective algebraic geometry and coding theory}. ArXiv:1506.02047.
\bibitem{clp} Croot, E.,  Lev, V.P., Pach, P. {\em Progression-free sets in $\mZ_4^n$ are exponentially small}. Ann. of Math., vol. 185, 331-337, (2017).
\bibitem{CM} Cohen, A., Moshkovitz, G. {\em Structure vs. randomness for bilinear maps}, to appear in the 53rd ACM Symposium on Theory of Computing (STOC 2021).
\bibitem{cm} Cook, B., Magyar, A. {\em Diophantine equations in the primes}. Inventiones mathematicae, 198 (3) 701-737, 2014. 
\bibitem{D} Dersken. H. {\em The $G$-stable rank for tensors} {arxiv.org/abs/2002.08435}.
\bibitem{hr} Hrushovski, E.  {\em The Elementary Theory of the Frobenius Automorphisms} arXiv: 0406514.
\bibitem{eg} Ellenberg, J, Gijswijt, D. {\em On large subsets of $\mF_q^n$ with no three-term arithmetic progression}. Ann. of Math., vol. 185, 339-343, (2017).
\bibitem{gt} Green, B., Tao, T. {\em The distribution of polynomials over finite fields, with applications to the Gowers norms}. Contrib. Discrete Math, 4(2):1-36, 2009.
\bibitem{gw} Gowers, T, Wolf, J. {\em Linear Forms and Higher-Degree Uniformity for Functions On $\mF_p^n$}. Geom. Funct. Anal.  21, 36-69 (2011).
\bibitem{janzer} Janzer, O. {\em Polynomial bound for the partition rank vs the analytic rank of tensors} Discrete Analysis 2020:7. 
\bibitem{kl}  Kaufman, T.,  Lovett, S.{\em Worst case to average case reductions for polynomials}. Proc. 49th Annual IEEE Symposium on Foundations of Computer Science (2008),166-175.
\bibitem{kz-approx} Kazhdan, D.,  Ziegler, T. {\em  Approximate cohomology}. Selecta Math. (N.S.) 24 (2018), no. 1, 499 - 509.
\bibitem{kz-extension}  Kazhdan, D.,Ziegler, T. {\em Properties of high rank subvarieties of affine spaces}, Geom. Funct. Anal. 30 (2020), 1063-1096.
\bibitem{kz-survey} Kazhdan, D., Ziegler, T.  {\em Applications of algebraic combinatorics to algebraic geometry}, arXiv:2005.12542 (2020). 
\bibitem{kz-uniform} Kazhdan, D., Ziegler, T. {\em Polynomials as splines}  Sel. Math. New Ser. (2019) 25: 31.
\bibitem{Kr} Kronecker, L. {\em Algebraische Reduction der Scharen bilinearer Formen}, Sitzungsber. Akad. Berlin, Jbuch. 22 (169)
(1890) 1225-1237.
\bibitem{lovett-rank} Lovett, S. {\em The Analytic rank of tensors and its applications}. Discrete Analysis 2019:7 
\bibitem{Mc} F. S. Macaulay, \emph{The algebraic theory of modular systems}, Cambridge University, 1916.
\bibitem{Mi} Mili\'cevi\'c, L.  {\em Polynomial bound for partition rank in terms of analytic rank} Geom. Funct. Anal. 29 (2019), no. 5, 1503-1530. 
\bibitem{neslund} Naslund, E. {\em The partition rank of a tensor andk-right corners in $\mF^n_q$}.  arXiv:1701.04475. (2017)
\bibitem{S} Schmidt, W.  M. {\em The density of integer points on homogeneous varieties}. Acta Math. 154 (1985), no. 3-4, 243-296. 
\bibitem{tt} Tao, T. {\em A symmetric formulation of the Croot-Lev-Pach-Ellenberg-Gijswijt capset bound}. terrytao.wordpress.com/2016/05/18/a-symmetric-formulation-of-the-croot-lev-pach-ellenberg-gijswijt-capset-bound/
\bibitem{W} Wooley, T. {\em An explicit version of Birch's Theorem}. Acta Arithmetica, 85(1), (1998) 79-96.
\bibitem{Lov} Lovett, S. {\em The analytic rank of tensors and its applications}. Discrete Analysis 2019:7, 10 pp. 
\end{thebibliography}
\end{document}